\pdfoutput=1
\documentclass[twoside,11pt]{article}

\usepackage{jmlr2e}

\hypersetup{
  implicit=false,
  citecolor=black,
  linkcolor=black,  
  colorlinks=true,
  urlcolor=blue
}
\usepackage[T1]{fontenc}      
\usepackage[utf8]{inputenc}   
\DeclareUnicodeCharacter{00A0}{ } 
\usepackage{microtype} 
\usepackage{commath}
\usepackage{mathtools}
\usepackage{thmtools}
\usepackage{thm-restate}
\usepackage[capitalise]{cleveref}

\makeatletter
\def\cleartheorem#1{%
    \expandafter\let\csname#1\endcsname\relax
    \expandafter\let\csname c@#1\endcsname\relax
}
\makeatother
\def\clearthms#1{ \@for\tname:=#1\do{\cleartheorem\tname} }
\cleartheorem{theorem}
\cleartheorem{example}
\cleartheorem{lemma}
\cleartheorem{proposition}
\cleartheorem{remark}
\cleartheorem{corollary}
\cleartheorem{definition}
\cleartheorem{conjecture}
\cleartheorem{axiom}

\numberwithin{equation}{section}
\declaretheorem[Refname={Theorem,Theorems}]{theorem}
\numberwithin{theorem}{section} 
\declaretheorem[style=definition,numberlike=theorem,Refname={Definition,Definitions}]{definition}
\declaretheorem[style=definition,numberlike=theorem,Refname={Assumption,Assumptions}]{assumption}

\declaretheorem[style=definition,numberlike=theorem,Refname={Remark,Remarks}]{remark}
\declaretheorem[numberlike=theorem,Refname={Lemma,Lemmas}]{lemma}

\declaretheorem[name=Proposition,numberlike=theorem,Refname={Proposition,Propositions}]{proposition}

\newcommand{\qed}{\hfill\BlackBox\\[2mm]}

\DeclareMathOperator{\e}{e} 

\DeclareMathOperator*{\argmin}{arg\,min}

\DeclarePairedDelimiterX\Set[2]{\lbrace}{\rbrace}%
{ #1 \,:\, #2 }                                         
\DeclarePairedDelimiterX\inprod[2]{\langle}{\rangle}%
{ #1 , #2 }                                             

\newcommand{\R}{\mathbb{R}} 
\newcommand{\N}{\mathbb{N}} 

\newcommand{\T}{\mathsf{T}} 

\newcommand{\ML}{\textup{ML}}
\newcommand{\CV}{\textup{CV}}

\newcommand{\Hel}{\textup{Hel}}
\newcommand{\MAP}{\textup{MAP}}

\DeclarePairedDelimiterX{\infdivx}[2]{(}{)}{%
  #1\;\delimsize\|\;#2%
}

\newcommand{\GP}{\textup{GP}}

\usepackage{lastpage}
\jmlrheading{24}{2023}{1-\pageref{LastPage}}{10/22}{2/23}{22-1153}{Toni Karvonen and Chris J.\ Oates}
\ShortHeadings{Maximum Likelihood Estimation is Ill-Posed}{Karvonen \& Oates}
\firstpageno{1}

\begin{document}

\title{Maximum Likelihood Estimation in Gaussian Process Regression is Ill-Posed}

\author{\name Toni Karvonen \email toni.karvonen@helsinki.fi \\
  \addr Department of Mathematics and Statistics \\
  University of Helsinki\\
  PL 56 (Pietari Kalmin katu 5) \\
  00014 Helsingin yliopisto, Finland
       \AND
       \name Chris J.\ Oates \email chris.oates@ncl.ac.uk \\
       \addr School of Mathematics, Statistics and Physics \\
       Newcastle University \\
       Newcastle upon Tyne, NE1 7RU, United Kingdom
}
\editor{Marc Peter Deisenroth}

\maketitle

\begin{abstract}
\noindent Gaussian process regression underpins countless academic and industrial applications of machine learning and statistics, with maximum likelihood estimation routinely used to select appropriate parameters for the covariance kernel.
However, it remains an open problem to establish the circumstances in which maximum likelihood estimation is well-posed, that is, when the predictions of the regression model are insensitive to small perturbations of the data.
This article identifies scenarios where the maximum likelihood estimator \emph{fails} to be well-posed, in that the predictive distributions are not Lipschitz in the data with respect to the Hellinger distance.
These failure cases occur in the noiseless data setting, for any Gaussian process with a stationary covariance function whose lengthscale parameter is estimated using maximum likelihood.
Although the failure of maximum likelihood estimation is part of Gaussian process folklore, these rigorous theoretical results appear to be the first of their kind.
The implication of these negative results is that well-posedness may need to be assessed \emph{post-hoc}, on a case-by-case basis, when maximum likelihood estimation is used to train a Gaussian process model.
\end{abstract}


\section{Introduction} \label{sec:intro}

Gaussian process regression is a popular tool used to construct a predictive model for a response variable as a function of one or more covariates of interest. 
As a strict generalisation of classical linear regression, and with the support of production-level software, Gaussian process regression has found myriad applications in both the academic and industrial contexts.
The success of Gaussian process regression, both in terms of predictive performance and quality of uncertainty quantification, is contingent on the use of a suitable covariance kernel~$K$ for the Gaussian process model.
This is often achieved by choosing $K$ from a parametric set~$\{K_\theta\}_{\theta \in \Theta}$ of candidate covariance kernels, with the parameter $\theta \in \Theta$ being selected based on the training data set.
The predictive performance of Gaussian processes is well-understood in a variety of asymptotic settings~\citep[e.g.,][]{Stein1999, Anderes2010, VaartZanten2011, Bachoc2017, Karvonen2020}.
However, not much is known about the non-asymptotic setting when~$\theta$ is estimated from a training data set.
In particular, it is an open problem to understand when the predictions from a Gaussian process model are \emph{well-posed}, in the sense that the predictions of the model are either continuous in the training data set or insensitive to small perturbations of the training data set.
Theoretical understanding of well-posedness is urgently needed to support the use of Gaussian process regression in sensitive applications, such as in mine gas safety monitoring~\citep{Dong2012}, malicious maritime activity detection~\citep{Kowalska2012}, and climate modelling~\citep{Revell2018}, where the reliability and robustness of predictions is critical.

Several methods exist to estimate $\theta$, including maximum likelihood estimation \citep{mardia1984maximum}, maximum \textit{a posteriori} estimation \citep[e.g.,][]{cunningham2008fast}, cross-validation \citep{geisser1979predictive}, Bayesian inference \citep{mackay1992bayesian}, kernel flows \citep{Chen2021}, and various bespoke approaches, for example when $\theta$ are the parameters of a neural network in a deep kernel~\citep{wilson2016deep}.
Among these, maximum likelihood estimators $\theta_\ML$ are arguably most widely used, for example being the default approach in Gaussian process software \citep{GPML,scikit-learn,roustant2012dicekriging,gpy2014,GPflow2017}.
In addition to statistics and machine learning, maximum likelihood is also occasionally used in the applied mathematical literature to construct a kernel interpolant; see \citet[Section~14.3]{FasshauerMcCourt2015} and \citet{Cavoretto2021} for recent examples.
Compared to other approaches, maximum likelihood is attractive due to the absence of any additional degrees of freedom (e.g., cross-validation requires a choice for how data are partitioned) and the possibility for automatic gradient-based optimisation.

Suppose that the data are modelled as being generated by a Gaussian process defined by a mean function $m$ and a positive-definite covariance kernel $K_\theta$.
For a noiseless training data set $Y = (y_1, \ldots, y_n) \in \R^{n}$, associated to a set $X$ of distinct covariates $x_1, \ldots, x_n \in \R^{d}$, a maximum likelihood estimator $\theta_\ML$ of $\theta$ satisfies
\begin{equation} \label{eq:mle-intro}
  \theta_\ML \in \argmin_{ \theta \in \Theta } \ell(\theta \mid Y) \:\: \text{ with } \:\:  \ell( \theta \mid Y) = Y_m^\T K_\theta(X, X)^{-1} Y_m + \log \det K_\theta(X, X),
\end{equation}
where $K_\theta(X, X) = (K_\theta(x_i, x_j))_{i,j=1}^n \in \R^{n \times n}$ is the positive-definite covariance matrix and $Y_m = (y_i - m(x_i))_{i=1}^n \in \R^n$.
See \citet[Section~6.4]{Stein1999} or \citet[Section~5.4.1]{RasmussenWilliams2006}.
Despite the simple form of the optimisation problem~\eqref{eq:mle-intro}, there is only limited understanding of the behaviour of $\theta_\ML$ in the \emph{deterministic interpolation} regime, where the data $y_i = f(x_i)$ are, in truth, generated from a fixed but unknown function $f \colon \R^d \to \R$.
An important open problem is to understand the behaviour of $\theta_\ML$ in terms of the data-generating function $f$ and the set of covariates, and the implications of this behaviour for predictions produced by the Gaussian process model.

The deterministic interpolation regime represents a simple but important instance of Gaussian process regression widely used in, for example, emulation of computer experiments~\citep{Sacks1989, KennedyOHagan2002}, probabilistic numerical computation~\citep{Diaconis1988, Cockayne2019, Hennig2022}, and Bayesian optimisation~\citep{Snoek2012}.
However, this regime is challenging to analyse, and results concerning maximum likelihood estimation appear limited to the asymptotic analyses.
\citet{XuStein2017}; \citet{Karvonen2020}; and \citet{Wang2021} exploited a closed form for the maximum likelihood estimator of a \emph{scale} parameter (i.e., $\theta = \{\sigma\}$ and $K_\theta = \sigma^2 K$) to analyse its behaviour as $n \to \infty$ in the fixed domain setting where the covariates $\{x_i\}_{i=1}^\infty$ are dense in a compact subset of $\R^d$.
In a similar manner, \citet{Karvonen2019-MLSP} analysed maximum likelihood estimation of the scale and \emph{lengthscale} parameters (i.e., $\theta = \{\sigma, \lambda\}$ and $K_\theta(x,x') = \sigma^2 K(x/\lambda , x' / \lambda)$) for a particular non-stationary Ornstein--Uhlenbeck process, while \citet{Karvonen2022} obtained asymptotic lower bounds on estimates of the smoothness parameter of the Matérn model.
The non-asymptotic behaviour of maximum likelihood estimation in the deterministic interpolation regime has yet to be studied.

\subsection{Contributions}

The principal contributions of this work are to demonstrate how the concept of well-posedness can be applied to Gaussian process interpolation and how rigorous theoretical analysis of well-posedness can be performed.
We prove that, in the deterministic interpolation regime, maximum likelihood estimation of a lengthscale parameter can fail to be well-posed.
We emphasise that this is a \emph{non-asymptotic} (i.e., the number $n$ of observations is kept fixed) result, in contrast to earlier work, and is based on the observation that the maximum likelihood estimate $\lambda_\textup{ML}$ of a lengthscale parameter $\lambda$ is infinite if the observations differ from the prior mean function by a constant vertical shift.

To be more precise, consider a Gaussian process with prior mean function $m$ and stationary covariance function of the form
\begin{equation} \label{eq:stationary-kernel-intro}
  K_\lambda(x, y) = \Phi\bigg( \frac{x-y}{\lambda} \bigg) \quad \text{ for } \quad x, y \in \R^d,
\end{equation}
where the lengthscale parameter $\lambda > 0$ determines the spatial correlation distance of the resulting Gaussian random field.
If the function $\Phi \colon \R^d \to \R$ is continuous, the covariance function~\eqref{eq:stationary-kernel-intro} tends pointwise to the constant $\Phi(0)$ as $\lambda \to \infty$.
Thus, if the data are approximately shifted from the mean function $m$ by a constant, it is intuitive that a large value will be taken by the maximum likelihood estimate~$\lambda_\ML$.
It is rigorously proven in this article (see \Cref{thm:main-theorem}) that if (a) $n \geq 2$ is fixed, (b) the function $\Phi$ in~\eqref{eq:stationary-kernel-intro} satisfies certain mild regularity conditions, and (c) there is a constant $c \in \R$ such that the data are $m$-\emph{constant} in that
\begin{equation} \label{eq:constant-data-intro}
  y_i = m(x_i) + c \quad \text{ for } \quad i = 1, \ldots, n,
\end{equation}
then $\lambda_\textup{ML} = \infty$.
This result can be viewed as a generalisation of the simple fact that maximum likelihood estimation fails if the data are fully explained by the prior mean: if $y_i = m(x_i)$ for every $i = 1,\ldots,n$, the first term of $\ell(\lambda \mid Y)$ in~\eqref{eq:mle-intro} is zero and thus the estimate of $\lambda$ must be infinite because $K_\lambda(X, X)$ tends to a singular matrix if and only if $\lambda \to \infty$.
From $\lambda_\textup{ML} = \infty$ it follows that the pointwise predictions produced by the fitted Gaussian process model assign all probability to a single point (see \Cref{thm:flat-limit}).
This phenomenon is undesirable, as it is not reasonable to claim infinite precision from a finite data set (which could, in this case, involve as few as $n=2$ values being observed).
Secondly, we prove that if the data are \emph{not} $m$-constant, then $\lambda_\ML < \infty$ so that the predictive distributions are non-degenerate.
Using these results we show in \Cref{sec:ill-posed} that maximum likelihood estimation is not well-posed in general, in the sense that the resulting predictive distributions are not Lipschitz in the data with respect to the Hellinger distance, which means that predictive inference can be sensitive to small perturbations of the data set.\footnote{Note that this notion of well-posedness of a parameter estimation method is stronger than that of its \emph{robustness}, defined by \citet[Section~3]{GuWangBerger2018} essentially as the impossibility of obtaining a singular covariance matrix.}

A constant mean shift in~\eqref{eq:constant-data-intro} is \emph{not} a pathological case in the deterministic interpolation context, though it does highlight one sense in which mathematical analysis may be easier when the data are assumed to come from a stochastic process, since a constant mean shift may be neglected as a measure zero event in that setting.
For example, Gaussian process regression has been used to explore \emph{discrepancy} between computer models \citep{brevault2020overview}, where a constant mean shift between the output of two computer models for the same phenomenon could reasonably be expected.
Similarly, in probabilistic numerical computation one could encounter a constant mean shift \citep[e.g., when modelling an integrand in Bayesian cubature, if that integrand is in fact constant;][]{briol2019probabilistic}, or in applications of Bayesian optimisation, where the data are obtained in a region where the objective function is constant.
The role of this article is therefore to highlight an important failure mode of maximum likelihood estimation in Gaussian process interpolation and, in doing so, to underscore the need for an improved theoretical understanding of parameter estimation in general.

It would be tempting to attribute these failings to the simplicity of the maximum likelihood estimator and the modelling choices that we consider, for surely something more sophisticated ought to render the problem well-posed.
\Cref{sec:what-does-not-help} demonstrates that there may not exist an easy solution in the deterministic interpolation regime, at least if the tractability of a Gaussian process model is to be retained.
Namely, we prove the following extensions:
\begin{enumerate}
\item[(i)] A certain cross-validation estimator of the scale parameter shares the undesirable property of producing infinite lengthscale estimates when the data are $m$-constant.
\item[(ii)] Inclusion of a parametric prior mean function, which too is estimated from the data (i.e., as in universal kriging), does not prevent ill-posedness.
\item[(iii)] Simultaneous maximum likelihood estimation of the scale and lengthscale parameters does not prevent ill-posedness.
\end{enumerate}

What \emph{does} guarantee well-posedness is the inclusion of a regularisation or a nugget term, which corresponds to an assumption by the user that the data are corrupted by additive Gaussian noise.
It is intuitive that the maximum likelihood estimator should be more well-behaved if the data-generating process is noisy as it may be merely by chance that the data set is $m$-constant.
Less intuitive is the numerical evidence in \Cref{sec:regularisation-nugget}, which indicates that $\lambda_\ML$ is infinite even in the regularised setting as long as the regularisation term is sufficiently large.
A possible interpretation of this observation is that for sufficiently large assumed noise level \emph{any} data could have been plausibly generated by a constant mean shift of the prior mean.
The use of regularisation can hardly be considered a proper solution to ill-posedness because the interpolation property of the conditional process, desirable in many applications, is lost and an influential degree of freedom is introduced.
As demonstrated in \Cref{sec:regularisation-prior}, another option for guaranteeing that the lengthscale estimates are always finite is to place a hyperprior on the lengthscale and use \emph{maximum a posteriori} estimation.
However, because the hyperprior determines the estimator, this approach is rather arbitrary and does not lend itself well to automation in software.

Most of our results apply to Matérn-type covariance functions whose Fourier transforms decay polynomially (see \Cref{assumption:sobolev-kernel}).
\Cref{sec:generalisation} discusses generalisations of this Fourier assumption, lengthscale estimation for product kernels, and the use of general linear information, such as derivative data.
Our proofs are predominantly based on reproducing kernel Hilbert space (RKHS) techniques and approximation theory in Sobolev spaces.
Complete proofs are relegated to \Cref{sec:proofs}.
No familiarity with RKHSs or Sobolev spaces is required to understand the statements of our main results and, outside the proofs, it is only in \Cref{sec:generalisation} that these concepts are used.
Practical and theoretical implications of our results are discussed in \Cref{sec:conclusion}.

\subsection{Related Literature}

For fixed kernel parameters $\theta$, the mathematical properties of Gaussian process interpolation as $n \to \infty$ can be deduced from the equivalent perspective of optimal interpolation in an RKHS.
See \citet{Fasshauer2011, Scheuerer2013}; and \citet{Kanagawa2018} for reviews on this equivalence.
The setting where the data-generating function $f$ is randomised has received considerable attention and a large number of results have been obtained. 
Consistency and asymptotic normality results for maximum likelihood estimation of scale and lengthscale parameters can be found in, for example, \citet{Ying1991, Zhang2004, Loh2005, Du2009, Anderes2010, Kaufman2013, Bachoc2013}; and \citet{Bevilacqua2019}.
The use of maximum likelihood to estimate \textit{smoothness} parameters (i.e., controlling the differentiability of the Gaussian process sample paths) has also been considered, notably by \citet{Szabo2015, Knapik2016, Chen2021}; and \citet{Karvonen2022}.
A well known issue with maximum likelihood estimation occurs when $\theta$ is non-identifiable, as can happen when one attempts to simultaneously estimate scale, lengthscale, and smoothness parameters in the Mat\'{e}rn covariance model~\citep{Zhang2004}.
However, a lack of identifiability does not affect the predictions produced by the model or the combinations of parameters that can be identified (the so-called \emph{microergodic} parameters).
These related works, whilst providing useful insight and powerful theoretical tools and techniques, do not apply in the non-asymptotic deterministic interpolation regime where $\theta$ is estimated using maximum likelihood.

\section{Maximum Likelihood Estimation Is Not Well-Posed} \label{sec:main}

Given any function $f$ and a set $X$ of points $x_1, \ldots, x_n$ in the domain of $f$, we use $f(X)$ to denote the $n$-vector $(f(x_1), \ldots, f(x_n))$.
All vectors in this article are to be understood to be in column format.

\subsection{Gaussian Process Interpolation}

In Gaussian process interpolation, the data are modelled as discrete and noiseless observations from a Gaussian process sample path.
A Gaussian process $f_\GP \sim \mathrm{GP}(m, K)$ on $\R^d$ is a stochastic process characterised by a mean function $m \colon \R^d \to \R$ and a covariance kernel $K \colon \R^d \times \R^d \to \R$:
\begin{equation*}
  \mathbb{E}[f_\GP(x)] = m(x) \quad \text{ and } \quad \mathrm{Cov}[ f_\GP(x), f_\GP(y) ] = K(x, y)
\end{equation*}
for all $x, y \in \R^d$.
Let $K(X, X) = (K(x_i, x_j))_{i,j=1}^n \in \R^{n \times n}$.
That $f_\GP$ is Gaussian means that all its finite-dimensional distributions are normally distributed, which is to say that
\begin{equation*}
  f_\GP(X) \sim \mathrm{N}( m(X), K(X, X) )
\end{equation*}
for any finite set of points $X \subset \R^d$.
Throughout this article the covariance kernel is assumed to be (strictly) \emph{positive-definite}, which means that
\begin{equation} \label{eq:pd-condition}
  \sum_{i=1}^n \sum_{j=1}^n a_i a_j K(x_i, x_j) > 0
\end{equation}
for any $n \in \N$, any non-zero vector $a = (a_1, \ldots, a_n)$, and any distinct points $x_i \in \R^d$.
This implies that the covariance matrix $K(X, X)$ is positive-definite and non-singular if $X$ consists of distinct points.
A covariance kernel is \emph{stationary} if there is a function $\Phi \colon \R^d \to \R$ such that
\begin{equation} \label{eq:stationary-kernel}
  K(x, y) = \Phi(x - y) \quad \text{ for all } \quad x, y \in \R^d.
\end{equation}
If $\Phi$ is to yield a positive-definite covariance kernel, it is necessary that $\Phi(0) > 0$ because otherwise the condition~\eqref{eq:pd-condition} fails for $n = 1$.
For the purposes of this article, the Matérn kernels constitute the most important class of stationary covariance functions.
Let $\sigma$, $\lambda$, and $\nu$ be positive.
A Matérn kernel with scale $\sigma$, lengthscale $\lambda$, and smoothness $\nu$ is given by
\begin{equation} \label{eq:matern}
  K(x, y) = \sigma^2 \frac{2^{1-\nu}}{\Gamma(\nu)} \bigg( \frac{\sqrt{2\nu} \norm[0]{x - y}}{\lambda} \bigg)^{\nu} \mathcal{K}_\nu \bigg( \frac{\sqrt{2\nu} \norm[0]{x - y}}{\lambda} \bigg),
\end{equation}
where $\Gamma$ is the Gamma function and $\mathcal{K}_\nu$ the modified Bessel function of the second kind.

Suppose that a noiseless training data set $Y = (y_1, \ldots, y_n) \in \R^n$ associated to a set $X$ of distinct covariates $x_1, \ldots, x_n \in \R^d$ has been obtained and let $Y_m = Y - m(X)$.
Define
\begin{equation*}
  K(x, X) = (K(x,x_i))_{i=1}^n \in \R^n \quad \text{ and } \quad K(X, y) = (K(x_i, y))_{i=1}^n \in \R^{n}
\end{equation*}
for any $x, y \in \R^d$.
The mean and covariance functions of the conditional Gaussian process $f_\GP \mid Y$ are obtained from the well known expressions
\begin{equation} \label{eq:conditional-mean}
  \mu(x) = \mathbb{E}[ f_\GP(x) \mid Y] = m(x) + K(x, X)^\T K(X, X)^{-1} Y_m
\end{equation}
and
\begin{equation} \label{eq:conditional-variance}
  P(x, y)^2 = \mathrm{Cov}[ f_\GP(x), f_\GP(y) \mid Y] = K(x, y) - K(x, X)^\T K(X, X)^{-1} K(X, y).
\end{equation}
We use the simplified notation $P(x, x)^2 = P(x)^2$ for the conditional variance.
It is often convenient, especially in our proofs, to use the function
\begin{equation} \label{eq:conditional-mean-s}
  s(x) = K(x, X)^\T K(X, X)^{-1} Y_m
\end{equation}
and write the conditional mean as $\mu = m + s$.
The assumption that there is no noise means that we are in the deterministic interpolation regime where $\mu$ is an interpolant to the data, which is to say that $\mu(x_i) = y_i$ and $P(x_i) = 0$ for every~$i = 1,\ldots,n$.
Crucially, this assumption allows us to leverage well known equivalences, reviewed in \Cref{sec:rkhs}, between Gaussian process interpolation and kernel-based minimum-norm interpolation~\citep[e.g.,][]{Scheuerer2013, Kanagawa2018}.

\subsection{Maximum Likelihood Estimation}

Maximum likelihood estimation is the most common method used to select parameters $\theta \in \Theta$ of a parametrised covariance kernel $K_\theta$.
Under the Gaussian process model $\mathrm{GP}(m, K_\theta)$, the probability density function of the data $Y$ given $\theta$ is~\citep[e.g.,][Section~5.4.1]{RasmussenWilliams2006}
\begin{equation} \label{eq:likelihood}
  \frac{1}{\det(2\pi K_\theta(X, X))^{1/2}} \exp\bigg(\!\!-\frac{1}{2} Y_m^\T K_\theta(X, X)^{-1} Y_m \bigg).
\end{equation}
Maximising~\eqref{eq:likelihood} over $\Theta$ is equivalent to minimising
\begin{equation} \label{eq:ell-function}
  \ell( \theta \mid Y ) = Y_m^\T K_\theta(X, X)^{-1} Y_m + \log \det K_\theta(X, X),
\end{equation}
which we call the \emph{modified log-likelihood function}, being the log-likelihood function up to subtraction and multiplication by negative constants.
That is, any maximum likelihood estimate (which may not be unique if $\ell$ is multimodal) of $\theta$ satisfies
\begin{equation*}
  \theta_\ML \in \argmin_{\theta \in \Theta} \ell( \theta \mid Y ).
\end{equation*}
The two terms that comprise the modified log-likelihood function~\eqref{eq:ell-function} are usually called the (negative) \emph{data-fit} and \emph{model complexity} terms, respectively.
Throughout the article we use subscripts to denote that various quantities depend on the kernel parameters.
For example, the parameter-dependent conditional mean and covariance are
\begin{align*}
  \mu_\theta(x) &= m(x) + K_\theta(x, X)^\T K_\theta(X, X)^{-1} Y_m, \\
  P_\theta(x, y)^2 &= K_\theta(x, y) - K_\theta(x, X)^\T K_\theta(X, X)^{-1} K_\theta(X, y).
\end{align*}

\begin{figure}

  \centering
  \includegraphics[width=\textwidth]{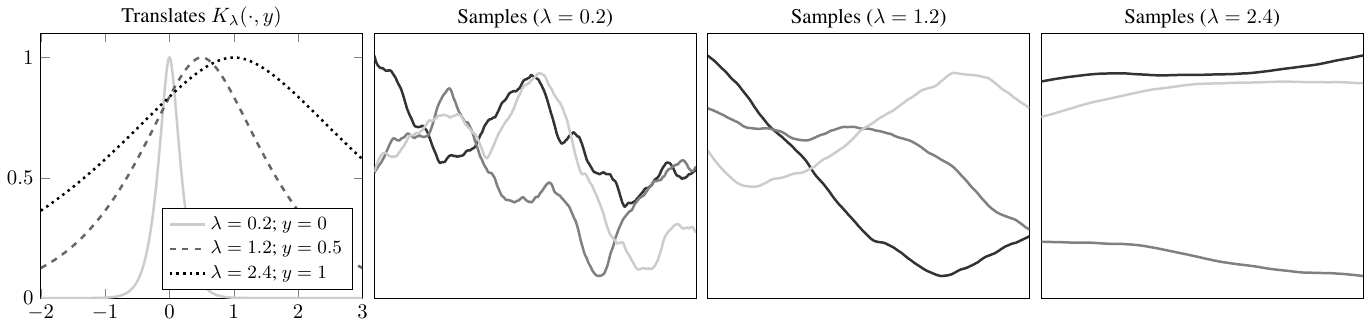}
  \caption{Translates of Matérn kernels in~\eqref{eq:matern}, as well as sample paths from the corresponding Gaussian processes on the domain $[0, 1]$, with $\sigma = 1$, $\nu = 3/2$, and three different~$\lambda$.} \label{fig:matern-samples}
\end{figure}

We are interested in estimation of the lengthscale parameter $\lambda > 0$ which parametrises any stationary kernel of the form~\eqref{eq:stationary-kernel} as
\begin{equation*}
  K_\lambda(x, y) = \Phi\bigg( \frac{x - y}{\lambda} \bigg)
\end{equation*}
and determines the spatial correlation distance of the resulting Gaussian process (see \Cref{fig:matern-samples}).
A maximum likelihood estimate of $\lambda$ therefore satisfies
\begin{equation*}
  \lambda_\ML \in \argmin_{ \lambda > 0 } \ell( \lambda \mid Y ) = \argmin_{ \lambda > 0 } \big\{ Y_m^\T K_\lambda(X, X)^{-1} Y_m + \log \det K_\lambda(X, X) \big\}.
\end{equation*}
Specifically, we are interested in rigorously proving that $\lambda_\ML = \infty$ in certain cases.
As~$\lambda$ increases, the kernel $K_\lambda$ tends pointwise to $\Phi(0)$, which is not positive-definite, and the covariance matrices in the predictive mean~\eqref{eq:conditional-mean} and covariance~\eqref{eq:conditional-variance} become singular if $n > 1$.
However, the limits $\lim_{\lambda \to \infty} \mu_\lambda$ and $\lim_{\lambda \to \infty} P_\lambda(x, y)^2$ do exist for most commonly used kernels, such as Matérns and the Gaussian, and it is in this limiting sense that one should interpret the conditional moments when $\lambda = \infty$.
This is discussed in more detail in \Cref{sec:flat-limit}.

Our results on the behaviour of $\lambda_\ML$ (and, later, other estimators of $\lambda$) show that Gaussian process interpolation fails, producing degenerate predictive distributions, when the data are $m$-constant and the function $\Phi$ has a polynomially decaying Fourier transform.

\begin{definition}[Constant data] \label{def:constant-data} Given a mean function $m$, we say that the data $Y$ are $m$-\emph{constant} if there is a constant $c \in \R$ such that
  \begin{equation*}
    Y_m = Y - m(X) = (c, \ldots, c) \in \R^n.
  \end{equation*}
\end{definition}

Let $\widehat{f}(\xi) = \int_{\R^d} g(x) \mathrm{e}^{-\mathrm{i} \xi^\T x} \dif x $ denote the Fourier transform of an integrable function $f \colon \R^d \to \R$.
We use the following assumption on the rate of decay of the Fourier transform of a stationary kernel.

\begin{assumption}[Stationary Sobolev kernel] \label{assumption:sobolev-kernel}
  There are a continuous and integrable function $\Phi \colon \R^d \to \R$ and constants $C_1$, $C_2 > 0$ and $\alpha > d/2$ such that $K(x, y) = \Phi(x - y)$ for all $x, y \in \R^d$ and
  \begin{equation} \label{eq:Phi-bounds}
    C_1 (1 + \norm[0]{\xi}^2)^{-\alpha} \leq \widehat{\Phi}(\xi) \leq C_2 (1 + \norm[0]{\xi}^2)^{-\alpha}
  \end{equation}
  for all $\xi \in \R^d$.
\end{assumption}

If $d=1$ and \Cref{assumption:sobolev-kernel} holds for $\alpha = p + 1 \in \N$, the kernel is $p$ times differentiable in that the derivative 
\begin{equation*}
  \frac{\partial^{2p}}{\partial x^{p} \partial y^p} K(x, y) \Bigr|_{\substack{x=0 \\ y=0}} = (-1)^{p} \Phi^{2p}(0)
\end{equation*}
exists.
As a consequence, the process $f_\GP \sim \mathrm{GP}(m, K)$ is $p$ times mean-square differentiable~\citep[Section~2.4]{Stein1999}.
That a kernel satisfying~\eqref{eq:Phi-bounds} is called a \emph{Sobolev kernel} is because its RKHS is norm-equivalent to the Sobolev space $W_2^\alpha(\R^d)$ of order $\alpha$.
The norm-equivalence is a crucial ingredient in several of our proofs and is reviewed, together with Sobolev spaces, in more detail in \Cref{sec:sobolev-spaces}.
One can also prove that the sample paths of $f_\GP$ are elements of certain Sobolev spaces~\citep{Scheuerer2011, Steinwart2019, Henderson2022}.
The Fourier transform of the function
\begin{equation*}
  \Phi(z) = \sigma^2 \frac{2^{1-\nu}}{\Gamma(\nu)} \big( \sqrt{2\nu} \norm[0]{z} \big)^{\nu} \mathcal{K}_\nu \big( \sqrt{2\nu} \norm{z} \big), \quad z \in \R^d,
\end{equation*}
which defines a Matérn kernel in~\eqref{eq:matern}, is~\citep[e.g.,][p.\@~49]{Stein1999}
\begin{equation} \label{eq:matern-fourier-transform}
  \widehat{\Phi}(\xi) = \sigma^2 \frac{\Gamma(\nu + d/2)}{ \pi^{d/2} \Gamma(\nu) } (2\nu)^\nu \big( 2\nu + \norm[0]{\xi}^2 \big)^{-(\nu+d/2)}.
\end{equation}
Therefore a Matérn kernel with smoothness $\nu > 0$ satisfies \Cref{assumption:sobolev-kernel} with $\alpha = \nu + d/2$.

With these preliminaries we are ready to state the main result of this article on the behaviour of maximum likelihood estimates of $\lambda$. The result is illustrated in \Cref{fig:log-likelihood-functions}.

\begin{restatable}[Maximum likelihood estimation]{theorem}{maintheorem} \label{thm:main-theorem}
  Suppose that the kernel $K$ satisfies \Cref{assumption:sobolev-kernel} and $n \geq 2$.
  If the data $Y$ are $m$-constant, then
  \begin{equation} \label{eq:mle-infinite}
    \lim_{ \lambda \to \infty} \ell(\lambda \mid Y) = -\infty \quad \text{ and } \quad \lambda_\ML = \infty.
  \end{equation}
  If the data $Y$ are not $m$-constant, then
  \begin{equation} \label{eq:mle-finite}
    \lim_{ \lambda \to \infty} \ell(\lambda \mid Y) = \infty \quad \text{ and } \quad \lambda_\ML < \infty.
  \end{equation}
\end{restatable}
\begin{proof}
  See \Cref{sec:proof-main}.
  The proof uses RKHS techniques to show that under \Cref{assumption:sobolev-kernel} the data-fit term, as a function of $\lambda$, (a)~is upper bounded if the data are $m$-constant and (b)~grows polynomially if the data are not $m$-constant, while the covariance matrix tends to the matrix consisting of $\Phi(0)$'s as $\lambda \to \infty$ and it can be shown that its log-determinant (i.e., the model complexity) tends to negative infinity with at most rate $-\log \lambda$.
\end{proof}

\begin{remark}
It would be very interesting and useful to obtain a more quantitative version of \Cref{thm:main-theorem} which would, for example, state that
\begin{equation} \label{eq:quantitative-bound}
  \lambda_\ML \geq g ( \mathrm{const}_m(Y) )
\end{equation}
for some measure $\mathrm{const}_m(Y)$ of how far $Y$ are from being $m$-constant (i.e., $\mathrm{const}_m(Y) = 0$ if and only if $Y$ are $m$-constant) and some decreasing function $g \colon (0, \infty) \to \R$ such that $g(r) \to \infty$ as $r \to 0$.
Unfortunately, the techniques we use to prove \Cref{thm:main-theorem} are not precise enough to prove any form of~\eqref{eq:quantitative-bound}.
\end{remark}

\begin{figure}

  \centering
  \includegraphics[width=\textwidth]{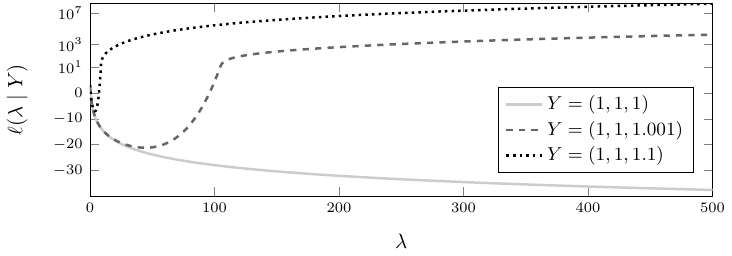}
  
  \caption{The modified log-likelihood function for the Matérn kernel~\eqref{eq:matern} with parameters $\sigma = 1$ and $\nu = 5/2$ given three different data vectors $Y$ obtained at the points $X = \{1, 1.2, 2\} \subset \R$. Note the non-linear $y$-axis.}
  \label{fig:log-likelihood-functions}
\end{figure}

\begin{remark}
  For simplicity, suppose that $m \equiv 0$ so that the data being $m$-constant means that $Y = (c, \ldots, c)$ for some $c \in \R$.
  Because the case $\lambda = \infty$ formally corresponds to a model with a constant kernel, one might be tempted to interpret \Cref{thm:main-theorem} as a special case of a general theorem which would state that $\theta_\ML$ takes the value $\bar{\theta}$ for which the data are fully explained by a single translate of $K_{\bar{\theta}}$.
  However, there can be no such theorem.
  For suppose that there are $\bar{\theta} \in \Theta$ and $i \in \{1, \ldots, n\}$ such that the data could have been generated by the translate of $K_{\bar{\theta}}$ at $x_i$. That is, $Y = a K_{\bar{\theta}}(x_i, X)$ for some $a \in \R$.
  Then
  \begin{equation*}
      \ell(\theta \mid Y) = Y^\T K_\theta(X, X)^{-1} Y + \log \det K_\theta(X, X)
  \end{equation*}
  would have to attain its minimum at $\theta = \bar{\theta}$.
  But because the data-fit term is non-negative and the model complexity term does not depend on the data, it is clear that a minimum can be attained at $\bar{\theta}$ only by ``chance'' or if $ \log \det K_{\bar{\theta}}(X, X) = -\infty$, which happens only if $K_{\bar{\theta}}(X, X)$ is singular or, in other words, if $K_{\bar{\theta}}$ is not a valid positive-definite kernel.
  That is, there can be no general theorem that $\theta_\ML = \bar{\theta}$ if it is required that $K_{\bar{\theta}}$ be a well-defined positive-definite kernel.
\end{remark}

Next we discuss the behaviour of the conditional mean and covariance and give a precise meaning to ill-posedness of Gaussian process interpolation that we have repeatedly alluded to.

\subsection{Conditional Mean and Variance in the Flat Limit} \label{sec:flat-limit}

Given \Cref{thm:main-theorem}, the question that arises is how the conditional mean~\eqref{eq:conditional-mean} and covariance~\eqref{eq:conditional-variance} behave if the data are $m$-constant.
As the kernel becomes constant for $\lambda = \infty$, the linear systems in the equations which define the conditional moments are singular if $n \geq 2$.
The sensible approach is therefore to consider the limits of $\mu_\lambda(x)$ and $P_\lambda(x, y)$ as $\lambda \to \infty$.
This \emph{flat limit} has been extensively studied during the past twenty years in the literature on radial basis function interpolation; see, for instance, \citet{Lee2015} or \citet{BarthelmeUsevich2021} and the references therein.
Flat limits have been recently considered in the context of Gaussian process interpolation by \citet{Barthelme2022}.
The conclusion of this body of research is that, under certain assumptions on the kernel and the covariate set geometry, the kernel-dependent term of the conditional mean in~\eqref{eq:conditional-mean-s} tends pointwise to (a)~a polynomial interpolant if the kernel is infinitely differentiable~\citep[Theorem~3.4]{LeeYoonYoon2007} or (b)~a polyharmonic spline interpolant if the kernel is finitely differentiable~\citep[Theorem~1]{Song2012}.
Interestingly, if the data are $m$-constant, we find that it is possible to present a simpler proof that is completely self-contained.

\begin{restatable}{theorem}{flatlimittheorem} \label{thm:flat-limit}
  Suppose that $K$ satisfies \Cref{assumption:sobolev-kernel} and $n \geq 1$.
  If $Y_m = (c, \ldots, c)$ for some $c \in \R$, then
  \begin{equation*}
    \lim_{ \lambda \to \infty } \mu_{\lambda}(x) = m(x) + c \quad \text{ and } \quad \lim_{\lambda \to \infty} \abs[0]{ P_\lambda(x, y) } = 0
  \end{equation*}
  for any $x, y \in \R^d$.
\end{restatable}
\begin{proof}
  See \Cref{sec:proofs-flat-limit}. 
\end{proof}

\Cref{thm:flat-limit} is illustrated in \Cref{fig:flat-limit}.
Since $\Phi$ is assumed continuous in \Cref{assumption:sobolev-kernel},  $\mu_\lambda(x)$ and $P_\lambda(x, y)$ are continuous functions of $\lambda$ for any fixed $x, y \in \R^d$ (which is proved similarly to \Cref{lemma:continuity}).
Therefore \Cref{thm:flat-limit} justifies writing
\begin{equation} \label{eq:flat-limit-interpretation}
  \mu_{\lambda = \infty}(x) = m(x) + c \quad \text{ and } \quad P_{\lambda = \infty}(x, y) = 0.
\end{equation}

\begin{figure}[t]
  \centering
  \includegraphics{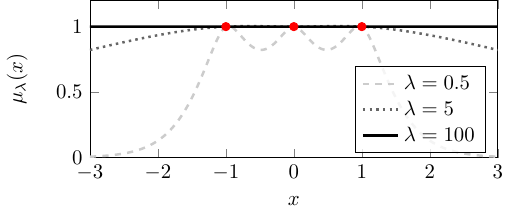}
  \caption{The conditional mean function for $\lambda \in \{1, 10, 100\}$ when $X = \{-1, 0, 1\} \subset \R$, $Y = (1, 1, 1)$, the prior mean is zero, and $K$ is the Matérn kernel~\eqref{eq:matern} with $\sigma = 1$ and $\nu = 3/2$.} \label{fig:flat-limit} 
\end{figure}

\subsection{Ill-Posedness} \label{sec:ill-posed}

By the classical definition of Hadamard, an inference or estimation problem is \emph{well-posed} if (i) a solution exists, (ii) the solution is unique, and (iii) the solution depends on continuously on the data.
If these conditions are not met, the problem is \emph{ill-posed}.
In the Bayesian inverse problems literature, where the solution is a posterior measure, the third condition is often strengthened to a requirement that the posterior be locally Lipschitz in the data with respect to the Hellinger distance~\citep[Section~4]{Stuart2010}.
One of the reasons that make Hellinger distance suitable in our context is that two distributions that are close in Hellinger distance are close also in mean and variance.
However, one may also consider other probability metrics~\citep[e.g.,][]{Latz2020}.
In this section we show that prediction using Gaussian process interpolation is not well-posed in the latter Lipschitz--Hellinger sense if the lengthscale parameter is set using maximum likelihood.
Bringing formal notions of well-posedness to bear on Gaussian process interpolation has not, to the best of our knowledge, previously been attempted.

Let $Q_1$ and $Q_2$ be two probability distributions on $\R^q$ that are absolutely continuous with respect a reference measure $\nu$ on $\R^q$ and let $q_1$ and $q_2$ denote their Radon--Nikodym derivatives with respect to $\nu$.
The squared Hellinger distance between $Q_1$ and $Q_2$ is
\begin{equation} \label{eq:Hellinger-definition}
  d_\Hel(Q_1, Q_2)^2  = \frac{1}{2} \int_{\R^q} \big( q_1(x)^{1/2} - q_2(x)^{1/2} \big)^2 \dif \nu(x).
\end{equation}
The Hellinger distance does not depend on the reference measure $\nu$, which means that for distributions that admit Lebesgue density functions we may set $\dif \nu(x) = \dif x$.
For univariate Gaussians $Q_1 = \mathrm{N}(\mu_1, \Sigma_1)$ and $Q_2 = \mathrm{N}(\mu_2, \Sigma_2)$, we have
\begin{equation} \label{eq:hellinger-gaussians}
  d_\Hel(Q_1, Q_2)^2 = 1 - \frac{\sqrt{2} (\Sigma_1 \Sigma_2)^{1/4}}{\sqrt{\Sigma_1 + \Sigma_2}} \exp\bigg( \!- \frac{(\mu_1 - \mu_2)^2}{4(\Sigma_1 + \Sigma_2)} \bigg).
\end{equation}
Let $Q(Y)$ stand for a posterior measure given an observed data vector $Y \in \R^n$.
The posterior is said to be well-posed if for every $\varepsilon > 0$ there exists $L > 0$ such that
\begin{equation} \label{eq:hellinger-lipschitz}
  d_\Hel\big( Q(Y), Q(Y') \big) \leq L \norm[0]{Y - Y'}
\end{equation}
for any data vectors $Y, Y' \in \R^n$ for which $\norm[0]{Y - Y'} \leq \varepsilon$.

Let us consider the Gaussian process predictive distribution at some unobserved point $x_0 \notin X$ as the posterior and set
\begin{equation} \label{eq:gp-predictive}
  Q_\GP(Y) = \mathrm{N}( \mu_{\lambda_\ML(Y)}(x_0), P_{\lambda_\ML(Y)}(x_0)^2 ),
\end{equation}
where we use $\lambda_\ML(Y)$ to denote that a maximum likelihood estimate depends on the data $Y$.
We may assume that $\lambda_\ML(Y)$ (or, if the modified log-likelihood function has multiple global minimum points, the largest of these) is a continuous function of the data, for otherwise predictions would not be continuous in the data, let alone Lipschitz.
Let $\varepsilon > 0$ and let $(Y_k)_{k=1}^\infty$ and $(Y_k')_{k=1}^\infty$ be two data sequences which satisfy $\norm[0]{Y_k - Y_k'} \leq \varepsilon$ for every $k \in \N$ and which converge to an $m$-constant data set:
\begin{equation*}
  \lim_{k \to \infty} Y_k - m(X) = \lim_{k \to \infty} Y_k' - m(X) = (c, \ldots, c) \in \R^n
\end{equation*}
for some $c \in \R$.
By \Cref{thm:main-theorem,thm:flat-limit} and the assumed continuity of $\lambda_\ML(Y)$ in the data, these sequences can be selected such that
\begin{equation*}
  \Sigma_k \coloneqq P_{\lambda_\ML(Y_k)}(x_0)^2 = C_1 \mathrm{e}^{-k} \quad \text{ and } \quad \Sigma_k' \coloneqq P_{\lambda_\ML(Y_k')}(x_0)^2 = C_2 k^{-1}
\end{equation*}
for some positive constants $C_1$ and $C_2$.
Since $\e^{-x} \leq 1$ for all $x \geq 0$, we get from~\eqref{eq:hellinger-gaussians} and~\eqref{eq:gp-predictive} that
\begin{equation*}
  \begin{split}
    d_\Hel ( Q_\GP(Y_k), Q_\GP(Y_k') )^2 \geq 1 - \frac{\sqrt{2} (\Sigma_k \Sigma_k')^{1/4}}{\sqrt{\smash[b]{\Sigma_k + \Sigma_k'}}} &= 1 - \frac{\sqrt{2} (C_1 C_2)^{1/4} k^{-1/4} \mathrm{e}^{-k/4}}{\sqrt{\smash[b]{C_1 \mathrm{e}^{-k} + C_2 k^{-1}}}} \\
    &\geq 1 - \sqrt{2} \, C_1^{1/4} C_2^{-1/4} k^{1/4} \mathrm{e}^{-k/4},
    \end{split}
\end{equation*}
where the second term tends to zero as $k \to \infty$.
Therefore 
\begin{equation*}
  d_\Hel ( Q_\GP(Y_k), Q_\GP(Y_k') ) \to 1 \quad \text{ as } \quad k \to \infty
\end{equation*}
even though $\norm[0]{Y_k - Y_k'} \to 0$ as $k \to \infty$.
This shows that the Lipschitz condition~\eqref{eq:hellinger-lipschitz} fails to hold when the data domain is 
\begin{equation*}
  \mathcal{R}^n = \Set{ Y \in \R^n }{ Y \text{ is not $m$-constant }} \subset \R^n,
\end{equation*}
the set of data sets that are not $m$-constant.
That is, we have shown that the mapping $Q_\GP \colon \mathcal{R}^n \to \mathcal{P}$ defined in~\eqref{eq:gp-predictive} is not Lipschitz, where $\mathcal{P}$ is the space of probability distributions on $\R$ equipped with the Hellinger distance.

The above derivation is a consequence of the fact that, from \Cref{thm:main-theorem} and the interpretation in~\eqref{eq:flat-limit-interpretation},
\begin{equation*}
  P_{\lambda_\ML(Y)}(x_0)^2 = 0 
\end{equation*}
if the data $Y$ are $m$-constant. Then for any data $Y'$ which are not $m$-constant we compute from~\eqref{eq:hellinger-gaussians} that
\begin{equation*}
  d_\Hel ( Q_\GP(Y), Q_\GP(Y') ) = 1,
\end{equation*}
which means that the predictive distribution is not continuous at any data which are $m$-constant.
Note that this is a purely formal computation because~\eqref{eq:Hellinger-definition} and~\eqref{eq:hellinger-gaussians} are valid only for measures which are absolutely continuous with respect a common reference measure, which is not the case with the degenerate Gaussian predictive distribution $Y$ that arises from $m$-constant data and the non-degenerate Gaussian $Y'$.
By observing that the argument above uses \Cref{thm:main-theorem} only to guarantee the existence of data $Y$ for which $\lambda_\ML(Y) = \infty$, we may formulate the following generic ill-posedness theorem.

\begin{theorem}[Ill-posedness] \label{thm:ill-posedness}
  Suppose that $K$ satisfies \Cref{assumption:sobolev-kernel} and $n \geq 1$.
  Let $\bar{\lambda} \colon \R^n \to [0, \infty]$ be any estimator of $\lambda$ and define $\mathcal{R}^n = \Set{ Y \in \R^n }{ \bar{\lambda}(Y) < \infty} \subset \R^n$.
  If there are data $Y \in \R^n$ such that $\bar{\lambda}(Y) = \infty$ (i.e., $\mathcal{R}^n \neq \R^n$), then Gaussian process interpolation is ill-posed, in the sense that the predictive distribution mapping $Q_\GP \colon \mathcal{R}^n \to \mathcal{P}$ defined in~\eqref{eq:gp-predictive} is not Lipschitz for any $x_0 \notin X$.
\end{theorem}

The main message of \Cref{thm:ill-posedness} is that a lengthscale estimator, whatever it might be, must be finite for any data in order for Gaussian process interpolation to be well-posed.

\section{What Does Not Help} \label{sec:what-does-not-help}

The Gaussian process model in \Cref{thm:main-theorem} is fairly simple, having a fixed prior mean function and a single estimated hyperparameter.
One might hope that additional modelling choices---or the use of an altogether different parameter estimation method---would yield a well-posed Gaussian process model.
In this section we show that this is not to be for several common approaches.
Each theorem in this section shows that an estimator of $\lambda$ is infinite if the data are $m$-constant, so that \Cref{thm:ill-posedness} consequently establishes that Gaussian process interpolation is ill-posed.

\subsection{Cross-Validation} \label{sec:cross-validation}

\emph{Leave-one-out cross-validation} is a popular alternative to maximum likelihood estimation that has been shown to confer robustness when the Gaussian process model is misspecified~\citep{Bachoc2013}.
In Gaussian process interpolation the objective function that is typically used is
\begin{equation} \label{eq:cv1-ell}
  \ell_{\CV}( \theta \mid Y ) = \sum_{k = 1}^n \bigg[ \bigg( \frac{ y_k - \mu_{\theta,n,k}(x_k) }{P_{\theta,n,k}(x_k)} \bigg)^2 + \log [P_{\theta,n,k}(x_k)^2] \bigg],
\end{equation}
where $\mu_{\theta,n, k} = m + s_{\theta, n, k}$ and $P_{\theta,n,k}$ denote the Gaussian process conditional mean and standard deviation functions in~\eqref{eq:conditional-mean}--\eqref{eq:conditional-mean-s} based on data at the points $X \setminus \{x_k\}$; see, for example, Section~4.2 in \citet{Currin1988} or Section~5.4.1 in \citet{RasmussenWilliams2006}.
Subscripts are again used to make explicit the dependency of these functions on the kernel parameters $\theta$.
Any corresponding parameter estimate $\theta_\CV$ satisfies
\begin{equation*}
  \theta_{\CV} \in \argmin_{ \theta \in \Theta } \ell_{\CV}( \theta \mid Y ).
\end{equation*}
The cross-validation objective function~\eqref{eq:cv1-ell} is obtained by summing negative predictive log-probabilities of $y_k$ given data at $X \setminus \{x_k\}$ and discarding terms which do not depend on~$\theta$.
Unfortunately, leave-one-out cross-validation also fails to be well-posed.
This may not be surprising given that there is a close connection between maximum likelihood estimation and cross-validation~\citep{FongHolmes2020}.

\begin{restatable}[Cross-validation]{theorem}{crossvalidation} \label{thm:cross-validation}
  Suppose that $K$ satisfies \Cref{assumption:sobolev-kernel} and $n \geq 2$.
  If the data~$Y$ are $m$-constant, then
  \begin{equation} \label{eq:cv-infinite}
    \lim_{ \lambda \to \infty} \ell_{\CV}(\lambda \mid Y) = -\infty \quad \text{ and } \quad \lambda_{\CV} = \infty.
  \end{equation}
\end{restatable}
\begin{proof}
  See \Cref{sec:proofs-what-does-not-help}. Despite the ostensibly different forms of the objective functions~\eqref{eq:ell-function} and~\eqref{eq:cv1-ell}, the proof is, in consequence of \Cref{prop:ell-representation}, in essence all but identical to the proof of \Cref{eq:mle-infinite} in \Cref{thm:main-theorem}.
\end{proof}

We believe that $\ell_\CV(\lambda \mid Y)$ and $\lambda_\CV$ satisfy a version of \Cref{eq:mle-finite} if the data are not $m$-constant but have been unable to furnish a proof; see \Cref{remark:cross-validation-proof-issue}.

\begin{remark} A non-probabilistic alternative to~\eqref{eq:cv1-ell} is to simply minimise the sum of squared leave-one-out errors~\citep[e.g.,][]{Rippa1999}:
\begin{equation*}
  \ell_{\CV(2)}( \theta \mid Y ) = \sum_{k=1}^n ( y_k - \mu_{\theta,n,k}(x_k) )^2 \geq 0 \quad \text{ and } \quad \theta_{\CV(2)} \in \argmin_{ \theta \in \Theta } \ell_{\CV(2)}( \theta \mid Y ).
\end{equation*}
Consider estimating the lengthscale parameter $\lambda$ using this procedure.
If the data $Y$ are $m$-constant such that $Y_m = Y - m(X) = (c, \ldots, c)$, it follows from \Cref{thm:flat-limit} that $\mu_{\lambda,n,k}$ tends pointwise to $m + c$ as $\lambda \to \infty$ if $K$ satisfies \Cref{assumption:sobolev-kernel}. Therefore
\begin{equation*}
  \lim_{ \lambda \to \infty } \ell_{\CV(2)}( \lambda \mid Y ) = \lim_{ \lambda \to \infty } \sum_{k=1}^n ( m(x_k) + c - \mu_{\lambda,n,k}(x_k) )^2 = 0,
\end{equation*}
from which it follows that $\lambda_{\CV(2)} = \infty$, or at least that $\ell_{\CV(2)}( \lambda \mid Y )$ has one of its minima at infinity.
Therefore also this procedure is ill-posed.
\end{remark}

\subsection{Unknown Parametric Prior Mean}

So far we have considered a setting where the prior mean function $m$ is known and fixed.
But in methods such as universal kriging the mean is assumed to be an unknown element of the linear span of a finite number of basis functions, typically polynomials, and its coefficients are estimated from the data.
See, for example, \citet{OHagan1978} or Chapters~3 and~4 in \citet{Santner2003}.

Let the basis functions be $\varphi_1, \ldots, \varphi_q$ for $q \leq n$ and define the matrix \sloppy{${V(X) \in \R^{n \times q}}$} with elements
\begin{equation*}
  (V(X))_{i,j} = \varphi_j(x_i).
\end{equation*}
Suppose that the mean function is $m = \sum_{j=1}^q \beta_j \varphi_j$ for unknown coefficients $\beta = (\beta_1, \ldots, \beta_q)$ which we wish to estimate using maximum likelihood.
The full modified log-likelihood function for both the kernel parameters $\theta$ and the coefficients $\beta$ is obtained by inserting the parametric prior mean in~\eqref{eq:ell-function}:
\begin{equation} \label{eq:ell-parametric-mean}
  \ell( \theta, \beta \mid Y ) = (Y - V(X) \beta)^\T K_\theta(X, X)^{-1} (Y - V(X) \beta) + \log \det K_\theta(X, X).
\end{equation}
Any maximum likelihood estimates satisfy
\begin{equation*}
  \{ \theta_\ML, \beta_\ML \} \in \argmin_{ \theta \in \Theta, \, \beta \in \R^q } \ell( \theta, \beta \mid Y).
\end{equation*}
The natural generalisation of \Cref{def:constant-data} to this setting is that there exist some coefficients for which the data are $m$-constant.
That is, that there exist constants $c$ and $\beta^* = (\beta_1^*, \ldots, \beta_q^*)$ (which need not be unique) such that
\begin{equation} \label{eq:constant-data-parametric}
  y_i - \sum_{j=1}^q \beta_j^* \varphi_j(x_i) = c \quad \text{ for every } \quad i = 1, \ldots, n.
\end{equation}
The next theorem shows that the maximum likelihood estimate of the lengthscale parameter~$\lambda$ is badly behaved if the data satisfy the above assumption.

\begin{restatable}{theorem}{parametrictheorem} \label{thm:parametric-prior-mean}
  Suppose that $K$ satisfies \Cref{assumption:sobolev-kernel} and $n \geq 2$.
  If the data $Y$ satisfy~\eqref{eq:constant-data-parametric}, then 
  \begin{equation*}
    \lambda_\ML = \infty \quad \text{ if } \quad \{ \lambda_\ML, \beta_\ML \} \in \argmin_{ \lambda > 0, \, \beta \in \R^q } \ell( \lambda, \beta \mid Y).
  \end{equation*}
\end{restatable}
\begin{proof}
  See \Cref{sec:proofs-what-does-not-help}. The proof is similar to that of \Cref{eq:mle-infinite} in \Cref{thm:main-theorem}.
\end{proof}

\begin{remark}
  If the matrix $V(X)$ has full rank, one can compute that
  \begin{equation*}
    \beta_\ML = \big[ V(X)^\T K_\theta(X, X)^{-1} V(X) \big]^{-1} V(X)^\T K_\theta(X, X)^{-1} Y
  \end{equation*}
  for any fixed $\theta$.
  If $q = n$, the matrix $V(X)$ is square and non-singular so that the above maximum likelihood estimate simplifies to $\beta_\ML = V(X)^{-1} Y$.
  Inserting this to~\eqref{eq:ell-parametric-mean} eliminates the data-fit term and we are left with $\ell(\theta, \beta_\ML \mid Y) = \log \det K_\theta(X, X)$.
  The maximum likelihood estimate of $\theta$ is therefore obtained by minimising complexity of the model.
  If $\theta = \lambda$, this naturally leads to $\lambda_\ML = \infty$ since this is the only value of the lengthscale parameter for which the covariance matrix becomes singular.
  The interpretation of this phenomenon is that maximum likelihood estimation picks the simplest possible model if the data are fully explained by the prior mean.
\end{remark}

\subsection{Simultaneous Estimation of the Scaling Parameter} \label{sec:sigma-lambda}

One typically estimates the lengthscale parameter simultaneously with the \emph{scale} or \emph{magnitude} parameter $\sigma > 0$.
Suppose that $\theta = \{\lambda, \sigma\}$ and the covariance kernel is parametrised as $K_\theta(x, y) = \sigma^2 K_\lambda(x, y)$.
Then
\begin{equation*}
  \ell( \lambda, \sigma \mid Y) = \frac{1}{\sigma^2} Y_m^\T K_\lambda(X, X)^{-1} Y_m + \log \det K_\lambda(X, X) + n \log \sigma^2
\end{equation*}
and maximum likelihood estimates satisfy
\begin{equation} \label{eq:mle-sigma-lambda}
  \{ \lambda_\ML, \sigma_\ML \} \in \argmin_{ \sigma, \lambda > 0} \ell(\lambda, \sigma \mid Y).
\end{equation}
Unfortunately, the behaviour of the maximum likelihood estimate $\lambda_\ML$ is identical to the case in \Cref{thm:main-theorem} where $\sigma$ is held fixed.
Let $\sigma_\ML(\lambda)$ denote the maximum likelihood estimate of $\sigma$ for a fixed $\lambda > 0$.

\begin{restatable}[Simultaneous estimation]{theorem}{simutheorem} \label{thm:simultaneous-estimation}
  Suppose that $K$ satisfies \Cref{assumption:sobolev-kernel} and $n \geq 2$.
  Consider the maximum likelihood estimates in~\eqref{eq:mle-sigma-lambda}.
  If the data $Y$ are $m$-constant and $Y_m \neq 0$, then
  \begin{equation*}
    \lim_{\lambda \to \infty} \ell(\lambda, \sigma_\ML(\lambda) \mid Y) = -\infty \quad \text{ and } \quad \lambda_\ML = \infty.
  \end{equation*}
\end{restatable}
\begin{proof}
  See \Cref{sec:proofs-what-does-not-help}. The proof does not fundamentally differ from that of \Cref{thm:main-theorem}.
\end{proof}

If it happens that $Y_m = 0$, the modified log-likelihood function is simply
\begin{equation*}
  \ell(\lambda, \sigma \mid Y) = \log \det K_\lambda(X, X) + n \log \sigma^2.
\end{equation*}
Then $\ell(\lambda, \sigma \mid Y) = -\infty$ if and only if $\lambda \to \infty$ or $\sigma \to 0$.
Either of these cases results zero conditional variance and degenerate predictive distributions.

\section{Regularisation} \label{sec:regularisation}

This section discusses two types of regularisation that can be used to ensure the well-posedness of Gaussian process regression or finiteness of a lengthscale estimator. However, these approaches may induce unwanted side effects and are, to some extent, arbitrary.

\subsection{Regularisation via Observation Noise} \label{sec:regularisation-nugget}

Let $\delta > 0$ be a \emph{regularisation parameter} (alternatively, \emph{smoothing parameter}, \emph{nugget} or \emph{jitter}).
Denote the $n \times n$ identity matrix with $I_n$.
The regularised versions of the Gaussian process conditional mean and covariance in~\eqref{eq:conditional-mean} and~\eqref{eq:conditional-variance} are
\begin{align*}
  \mu^\delta(x) &= m(x) + K(x, X)^\T (K(X, X) + \delta^2 I_n)^{-1} Y_m, \\
  P^\delta(x, y)^2 &= K(x, y) - K(x, X)^\T (K(X, X) + \delta^2 I_n)^{-1} K(X, y)
\end{align*}
and that of the modified log-likelihood function is
\begin{equation} \label{eq:ell-regularised}
  \ell^\delta( \lambda \mid Y ) = Y_m^\T (K_\lambda(X, X) + \delta^2 I_n)^{-1} Y_m + \log \det (K_\lambda(X, X) + \delta^2 I_n).
\end{equation}
In this section every quantity which is superscripted with $\delta$ stands for the corresponding quantity defined in \Cref{sec:main} but with $K(X, X)$ replaced by $K(X, X) + \delta^2 I_n$.
Although the above expressions arise from assuming that the data are corrupted by additive and independent zero-mean Gaussian noise terms with variances $\delta^2$, regularisation is often used purely out of convenience (that is, even when one does not believe that the data are noisy) as it improves the condition number of the covariance matrix that needs to be inverted~\citep[e.g.,][]{Ranjan2011, AdrianakisChallenor2012}.

Because $(K(X, X) + \delta^2 I_n)^{-1} < K(X, X)^{-1}$ in the Loewner ordering of positive-semidefinite matrices, we easily derive that the regularised conditional variance is everywhere positive:
\begin{equation*}
    P^\delta(x)^2 = P^\delta(x, x)^2 > P(x)^2 \geq 0 \quad \text{ for every } \quad x \in \R^d.
\end{equation*}
Consider then the lengthscale parameter $\lambda$ and the corresponding kernel $K_\lambda$ and assume for simplicity that $K_\lambda(x, y) \to 1$ as $\lambda \to \infty$ for all $x, y \in \Omega$.
Then it is straightforward to use the Sherman--Morrison formula to compute that
\begin{equation*}
    \lim_{\lambda \to \infty} P_\lambda^\delta(x)^2 = \lim_{\lambda \to \infty} \big[ K_\lambda(x, y) - K_\lambda(x, X)^\T (K_\lambda(X, X) + \delta^2 I_n)^{-1} K_\lambda(x, X) \big] = \frac{1}{1 + n \delta^{-2}}
\end{equation*}
and
\begin{equation*}
  \lim_{ \lambda \to \infty } \mu_\lambda^\delta(x) = m(x) + \frac{\delta^{-2} }{1 + n \delta^{-2}} \sum_{i=1}^n (y_i - m(x_i))
\end{equation*}
for any $x \in \R^d$ and $Y_m \in \R^n$.
These estimates and computations establish that under regularisation no finite-dimensional distribution of the conditional Gaussian process can tend to a degenerate Gaussian as $\lambda$ varies.
It follows that regularised Gaussian process interpolation is well-posed in the sense discussed in \Cref{sec:ill-posed} as long as the estimator of $\lambda$ is continuous in the data.
However, the price one pays for vanquishing ill-conditioning is that the conditional mean no longer interpolates the data, which may be undesirable in applications where the data are truly noiseless, and that an additional degree of freedom is introduced.
Favorable convergence rates of Gaussian process interpolation (asymptotically as $n \to \infty$) are also lost under regularisation unless one has the regularisation parameter tend to zero with an appropriate rate~\citep[Section~3]{WendlandRieger2005}.

\begin{figure}[t]
  \centering
  \includegraphics[width=\textwidth]{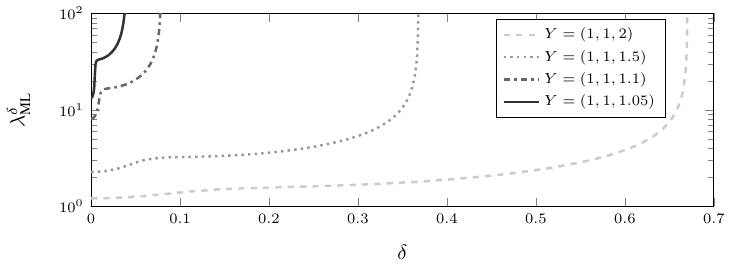}
  \caption{The regularised maximum likelihood estimate $\lambda_\ML^\delta$ in~\eqref{eq:lambda-ml-delta} as a function of the regularisation parameter $\delta$ for four different data sets $Y \in \R^3$.} \label{fig:regularisation}
\end{figure}

Even though maximum likelihood estimation of $\lambda$ cannot cause ill-posedness in the regularised setting (assuming the estimator is continuous in the data), there remains the interesting question of how $\lambda_\ML^\delta$, as computed by minimising the regularised modified log-likelihood function in~\eqref{eq:ell-regularised}, behaves if the data are $m$-constant.
Our attempts at proving any properties of $\lambda_\ML^\delta$ have been unsuccessful, and therefore here we limit ourselves to a simple numerical investigation, the results of which are depicted in \Cref{fig:regularisation}.
This figure plots
\begin{equation} \label{eq:lambda-ml-delta}
  \lambda_\ML^\delta = \argmin_{ \lambda > 0 } \ell^\delta( \lambda \mid Y )
\end{equation}
as a function of $\delta > 0$ for four different data vectors $Y \in \R^3$ (we set $m \equiv 0$) when $X = \{1, 1.2, 2.0\} \subset \R$ and $K$ is the Matérn kernel in~\eqref{eq:matern} with parameters $\sigma = 1$ and $\nu = 3/2$.
Minimisation was performed using grid search.
It appears that for each data set the maximum likelihood estimate $\lambda_\ML^\delta$ has a singularity at a certain value $\delta_\infty$ of $\delta$ and that $\delta_\infty$ is smaller when the data $Y$ are closer to being constant.
This suggests that, for a sufficiently large $\delta$, maximum likelihood estimation always reverts to the simplest possible model, that obtained with $\lambda = \infty$.
One way to interpret this observation is that any given data set could have been plausibly generated by a constant function if the data set is assumed to be corrupted by Gaussian noise with sufficiently large variance (i.e., if $\delta^2$ is sufficiently large).

\subsection{Regularisation via Lengthscale Hyperprior} \label{sec:regularisation-prior}

Suppose that a prior with a density function $p$ is placed on the parameters $\theta$.
Then the posterior for $\theta$ has the log-density
\begin{equation*}
  \log p(\theta \mid Y) = -\frac{1}{2} \ell( \theta \mid Y) + \log p(\theta) + \text{constant},
\end{equation*}
where $\ell(\theta \mid Y)$ is the modified log-likelihood function in~\eqref{eq:ell-function}.
Any maximiser $\theta_\MAP$ of the posterior density is called a \emph{maximum a posteriori} (MAP) estimate of $\theta$.
Equivalently,
\begin{equation} \label{eq:map-estimator}
  \theta_\MAP \in \argmin_{ \theta \in \Theta } \ell_\MAP( \theta \mid Y ) \quad \text{ for } \quad \ell_\MAP(\theta \mid Y) = \frac{1}{2} \ell( \theta \mid Y ) - \log p(\theta).
\end{equation}
The following theorem shows that assigning a non-heavy-tailed hyperprior on the lengthscale~$\lambda$ ensures that its MAP estimate is finite regardless of the data.

\begin{restatable}{theorem}{maptheorem} \label{thm:map-estimation}
  Suppose that $K$ satisfies \Cref{assumption:sobolev-kernel}, $n \geq 1$, and
  \begin{equation} \label{eq:hyperprior-tail-assumption1}
    \lambda^{(\alpha - d/2)n} p(\lambda) \to 0 \quad \text{ as } \quad \lambda \to \infty,
  \end{equation}
  where $\alpha > d/2$ is the constant in \Cref{assumption:sobolev-kernel}.
  If $Y \in \R^n$ are any data, then
  \begin{equation*}
    \lim_{\lambda \to \infty} \ell_\MAP( \lambda \mid Y ) = \infty \quad \text{ and } \quad \lambda_\MAP < \infty.
  \end{equation*}
\end{restatable}
\begin{proof}
  See \Cref{sec:proof-map-estimation}.
  The essence of the proof is that the assumption on tail decay of $p$ ensures that $-\log p(\lambda)$ dominates $\ell_\MAP( \lambda \mid Y)$ as $\lambda \to \infty$.
\end{proof}

Because, for fixed $Y$, the choice of the hyperprior $p$ completely determines the estimator, MAP estimation is rather arbitrary and not well-suited to deployment in general-purpose software.

\section{Generalisations and Extensions} \label{sec:generalisation}

This section discusses some generalisations of \Cref{thm:main-theorem} for (a) general linear data, (b) lengthscale estimation for product kernels, and (c) kernels which do not satisfy the Fourier decay assumption in~\eqref{eq:Phi-bounds}.

\subsection{Linear Information and General Kernels} \label{sec:linear-information}

In this section we generalise part of \Cref{thm:main-theorem} by replacing the domain $\R^d$ with an arbitrary vector space $\Omega$, using a more liberal definition of a lengthscale parameter, and considering general linear data, such as derivative evaluations.
Though somewhat technical, the assumptions that this generalisation requires can be verified in some settings of practical interest.

Let $\Omega$ be a vector space and $K_\theta \colon \Omega \times \Omega \to \R$ a positive-definite kernel on $\Omega$ for each $\theta \in \Theta$.
We use $H(K_\theta, \Omega)$ to denote the RKHS of $K_\theta$ on $\Omega$; see \Cref{sec:rkhs} for details.
Let $F(\Omega)$ be a set of real-valued functions defined on $\Omega$ which contains $H(K_\theta, \Omega)$ for every $\theta \in \Theta$ as well as all constant functions and the prior mean function $m$.
Let $\mathcal{L} = \{L_1, \ldots, L_n\}$ be a collection of $n \geq 2$ non-trivial (i.e., none of them is the zero functional) linear \emph{information functionals} defined on $F(\Omega)$. These functionals are assumed to be linearly independent and bounded on $H(K_\theta, \Omega)$ for every $\theta \in \Theta$, in that there is $C_\theta > 0$ such that $\abs[0]{L_i(f)} \leq C_\theta \norm[0]{f}_{H(K_\theta,\Omega)}$ for every $f \in H(K_\theta, \Omega)$.
We assume that the data
\begin{equation} \label{eq:general-data}
  Y_\mathcal{L} = ( L_1 f, \ldots, L_n f) \in \R^n
\end{equation}
consist of applications of the information functionals to an underlying (and unknown) data-generating function $f \in F(\Omega)$.
The setting considered earlier in this article is recovered by selecting the point evaluation functionals defined as $L_i f = f(x_i)$.
Partial derivative data, defined by $L_i f = \partial_{p_i} f(x_{i})$ for some $x_i \in \Omega$ and $p_i \geq 0$, also occurs commonly in Gaussian process applications~\citep[e.g.,][]{Solak2002}.
If necessary to avoid ambiguity, we use superscripts to denote the argument with respect to which an information functional is to be applied.

Let $\mathcal{L}(m) = (L_i m)_{i=1}^n \in \R^n$ and set $Y_{\mathcal{L},m} = Y_\mathcal{L} - \mathcal{L}(m) \in \R^n$.
Also set
\begin{equation*}
  K_\theta(\mathcal{L}, \mathcal{L}) = (L_i^x L_j^y K_\theta(x, y))_{i,j=1}^n \in \R^{n \times n} \quad \text{ and } \quad K_\theta(\mathcal{L}, x) = (L_i K_\theta(\cdot ,x))_{i=1}^n \in \R^n,
\end{equation*}
which are well-defined by the assumption that the information functionals are bounded on $H(K_\theta, \Omega)$.
The vector $K_\theta(x, \mathcal{L})$ is defined analogously to $K_\theta(\mathcal{L}, x)$ but with the information functionals applied to the second argument.
The Gaussian process $f_\textup{GP} \sim \mathrm{GP}(m, K_\theta)$ conditioned on the general linear data $Y_\mathcal{L}$ in~\eqref{eq:general-data} has the mean
\begin{equation*}
   \mathbb{E} [ f_\textup{GP}(x) \mid Y_\mathcal{L} ] = Y_{\mathcal{L}, m}^\T K_\theta(\mathcal{L}, \mathcal{L})^{-1} K_\theta(\mathcal{L}, x),
\end{equation*}
and covariance
\begin{equation*}
  \mathrm{Cov}[ f_\textup{GP}(x), f_\textup{GP}(y) \mid Y_\mathcal{L} ] = K_\theta(x, y) - K_\theta(x, \mathcal{L})^\T K_\theta(\mathcal{L}, \mathcal{L})^{-1} K_\theta(\mathcal{L}, y).
\end{equation*}
The modified log-likelihood function is
\begin{equation} \label{eq:ell-general}
  \ell( \theta \mid Y_\mathcal{L} ) = Y_{\mathcal{L}, m}^\T K_\theta( \mathcal{L}, \mathcal{L} )^{-1} Y_{\mathcal{L}, m} + \log \det K_\theta( \mathcal{L}, \mathcal{L} ).
\end{equation}

Let $g \colon (0, \infty) \to (0, \infty)$ be a continuous function such that (a) $\limsup_{\lambda \to 0} g(\lambda) < \infty$, (b) $\lim_{\lambda \to \infty} g(\lambda) = \infty$, and (c) $g(\lambda_1) = 1$ for some $\lambda_1 > 0$.
Define
\begin{equation} \label{eq:kernel-general}
  K_\lambda(x, y) = K\bigg( \frac{x}{g(\lambda)}, \frac{y}{g(\lambda)} \bigg) \quad \text{ for any } \quad \lambda > 0 \: \text{ and } \: x, y \in \Omega.
\end{equation}
The parameter $\lambda$ can be interpreted as a generalised version of the lengthscale parameter of a stationary kernel in~\eqref{eq:stationary-kernel-intro}.
Stationary kernels are recovered from~\eqref{eq:kernel-general} by setting $g(\lambda) = \lambda$ and $K(x,y) = \Phi(x -y)$.
The exponential kernel $K_\lambda(x, y) = \exp( xy /\lambda )$ is a simple non-stationary example that is occasionally used.
We are interested in maximum likelihood estimates $\lambda_{\ML(\mathcal{L})}$ of $\lambda$, any of which satisfies
\begin{equation*}
  \lambda_{\ML(\mathcal{L})} \in \argmin_{ \lambda > 0} \ell( \lambda \mid Y_\mathcal{L} ) = \argmin_{ \lambda > 0} \big\{ Y_{\mathcal{L}, m}^\T K_\lambda( \mathcal{L}, \mathcal{L} )^{-1} Y_{\mathcal{L}, m} + \log \det K_\lambda( \mathcal{L}, \mathcal{L} ) \big\}.
\end{equation*}
In this setting we say that the data $Y_\mathcal{L}$ are $m$-constant if there is a constant function $c \colon \Omega \to \R$ such that
\begin{equation} \label{eq:m-constant-general}
  Y_{\mathcal{L}, m} = ( L_1 c, \ldots, L_n c).
\end{equation}
To prove a generalisation of \Cref{thm:main-theorem} we need the following technical assumption.

\begin{assumption} \label{ass:general}
  Every element of the matrix $K_\lambda( \mathcal{L}, \mathcal{L} )$ is a continuous function of $\lambda > 0$ and
  \begin{equation*}
    \text{(a)} \quad \liminf_{ \lambda \to 0 } e_\textup{min} ( K_\lambda(\mathcal{L}, \mathcal{L} ) ) > 0 \quad \text{ and } \quad \text{(b)} \quad \lim_{ \lambda \to \infty } \det K_\lambda(\mathcal{L}, \mathcal{L} ) = 0,
  \end{equation*}
  where $e_\textup{min}(A)$ denotes the smallest eigenvalue of a matrix $A$.
\end{assumption}

\begin{restatable}{theorem}{gentheorem} \label{thm:generalisation}
  Suppose that Assumption~\ref{ass:general} holds and let $\Omega_b$ be any convex subset of $\Omega$ such that (a) $0 \in \Omega_b$ and (b) $Lf_1 = Lf_2$ for every $L \in \mathcal{L}$ and all $f_1, f_2 \in F(\Omega)$ such that $f_1 = f_2$ on $\Omega_b$.
  If the data $Y_\mathcal{L}$ are $m$-constant and constant functions are contained in $H(K, \Omega_b)$, then
  \begin{equation*}
    \lim_{ \lambda \to \infty }\ell( \lambda \mid Y_\mathcal{L} ) = -\infty \quad \text{ and } \quad \lambda_{\ML(\mathcal{L})} = \infty.
  \end{equation*}
\end{restatable}
\begin{proof}
  See \Cref{sec:proof-of-generalisation}. The proof is in essence identical to that of~\eqref{eq:mle-infinite}, but more technical.
\end{proof}

\Cref{ass:general} and other assumptions in \Cref{thm:generalisation} may be verified on case-by-case basis.
For example, let $\Omega = \R$ and
\begin{equation*}
  K_\lambda(x, y) = \bigg( 1 + \frac{\sqrt{3}\abs[0]{x-y}}{\lambda} \bigg) \exp\bigg(\! - \frac{\sqrt{3}\abs[0]{x-y}}{\lambda} \bigg),
\end{equation*}
which is the Matérn kernel in~\eqref{eq:matern} with smoothness $\nu = 3/2$.
Suppose that $L_1 f = f'(x_1)$ and $L_if = f(x_i)$ for $i = 2,\ldots,n$ and distinct $x_2, \ldots, x_n$.
One may easily find a convex $\Omega_b \subset \R$ such that $0,x_1,x_2,\ldots,x_n \in \Omega_b$ (e.g., a ball centered at the origin with radius that exceeds the maximal norm of the points).
Because this set can be taken to be bounded, it follows from~\eqref{eq:matern-fourier-transform} and the results reviewed in \Cref{sec:sobolev-spaces} that constant functions are contained in $H(K, \Omega_b)$.
To verify Assumption~\ref{ass:general}, observe that the generalised covariance matrix is
\begin{equation} \label{eq:kernel-matrix-derivatives}
  K_\lambda(\mathcal{L}, \mathcal{L}) = \begin{pmatrix} a_\lambda & b_\lambda^\T \\ b_\lambda & K_\lambda( X', X') \end{pmatrix},
\end{equation}
where $K_\lambda(X', X') \in \R^{(n-1) \times (n-1)}$ is the regular covariance matrix for the points $x_2, \ldots, x_n$,
\begin{equation*}
  a_\lambda = \frac{\partial^2}{\partial x \partial y} K_\lambda(x, y) \biggl|_{\substack{x = x_1 \\ y = x_1}} = \frac{3}{\lambda^2},
\end{equation*}
and
\begin{equation*}
  (b_\lambda)_{i-1} = \frac{\partial}{\partial x} K_\lambda(x, x_i) \biggl|_{x = x_1} = -\frac{3}{\lambda^2} (x_1 - x_i) \exp\bigg( \! -\frac{\sqrt{3} \abs[0]{x_1 - x_i}}{\lambda} \bigg)
\end{equation*}
for $i = 2,\ldots,n$.
The entries of $b_\lambda$, as well as $a_\lambda$, are continuous in $\lambda$ on $(0, \infty)$.
We compute $\lim_{ \lambda \to 0} a_\lambda = \infty$ and $\lim_{ \lambda \to \infty} a_\lambda = 0$, as well as $\lim_{ \lambda \to 0} (b_\lambda)_{i-1} = 0$ and $\lim_{ \lambda \to \infty} (b_\lambda)_{i-1} = 0$ for each $i=2,\ldots,n$.
Therefore $\lim_{ \lambda \to 0} K_\lambda(\mathcal{L}, \mathcal{L} ) = \mathrm{diag}(\infty, 1, \ldots, 1)$ and $\lim_{ \lambda \to \infty} K_\lambda(\mathcal{L}, \mathcal{L} )$ is singular because its first row is zero.
Assumption~\ref{ass:general} thus holds and it follows from \Cref{thm:generalisation} that $\lambda_{\textup{ML}(\mathcal{L})} = \infty$ if $Y_{\mathcal{L},m} = (0, c, \ldots, c)$ for some $c \in \R$.

\subsection{Product Kernels and Multiple Lengthscales} \label{sec:multiple-lengthscales}

Here we consider a setting where the covariance kernel is a product of stationary kernels equipped with dimensionwise lengthscale parameters.
That is, $\theta = \{\lambda_1, \ldots, \lambda_d\}$ and the kernel has the product form
\begin{equation} \label{eq:product-kernel}
  K_\theta(x, y) = \prod_{i=1}^d K_{i,\lambda_i}(x_i, y_i) = \prod_{i=1}^d \Phi_i \bigg( \frac{x_i - y_i}{\lambda_i} \bigg),
\end{equation}
where $x = (x_1, \ldots, x_d) \in \R^d$, $y = (y_1, \ldots, y_d) \in \R^d$ and $K_i(x_i, y_i) = \Phi_i(x_i - y_i)$ are stationary kernels on $\R$ parametrised by positive lengthscale parameters $\lambda_i$.
Recall the definition of Mat\'ern kernel from~\eqref{eq:matern}.
The product Mat\'ern kernel
\begin{equation*}
  K_\theta(x, y) = \sigma^2 \prod_{i=1}^d \Bigg[ \frac{2^{1-\nu_i}}{\Gamma(\nu_i)} \bigg( \frac{\sqrt{2\nu_i} \abs[0]{x_i - y_i}}{\lambda_i} \bigg)^{\nu_i} \mathcal{K}_{\nu_i} \bigg( \frac{\sqrt{2\nu_i} \abs[0]{x_i - y_i}}{\lambda_i} \bigg) \Bigg],
\end{equation*}
where $\nu_i > 0$, is a commonly used product kernel of the form~\eqref{eq:product-kernel}.

For simplicity, let us consider maximum likelihood estimation of only one of the $d$ lengthscale parameters.
For $p \in \{1, \ldots, d\}$, we are interested in the behaviour of
\begin{equation} \label{eq:mle-lengthscale-p}
  \lambda_{p,\ML} = \argmin_{ \lambda_p > 0} \ell( \lambda_1, \ldots, \lambda_d \mid Y),
\end{equation}
where $\ell( \lambda_1, \ldots, \lambda_d \mid Y)$ is the modified log-likelihood function in~\eqref{eq:ell-function} for $\theta = \{\lambda_1, \ldots, \lambda_d\}$ and the product kernel in~\eqref{eq:product-kernel} and $\lambda_i$ for $i \neq p$ are fixed.
For product covariates $X$ of the form
\begin{equation} \label{eq:product-covariates}
  X = X_1 \times \cdots \times X_d, \quad \text{ where } \quad X_i = \{x_{i,1}, \ldots, x_{i,n_i}\} \subset \R,
\end{equation}
we say that the associated data $Y$ are $m$-constant \emph{along dimension} $p$ if
\begin{equation*}
  y_{i_1, \ldots, i_d} - m(x_{1,i_1}, \ldots, x_{d, i_d}) \quad \text{ does not depend on } \quad i_p = 1, \ldots, n_p,
\end{equation*}
where the datum $y_{i_1, \ldots, i_d}$ is associated with the covariate $(x_{1,i_1}, \ldots, x_{d, i_d}) \in X$.
For example, the data
\begin{equation*}
  Y = (y_{1,1}, y_{1,2}, y_{1,3}, y_{2,1}, y_{2,2}, y_{2,3}) = (0, 1, 2, 0, 1, 2)
\end{equation*}
are constant along dimension $p = 1$ for the product design
\begin{equation*}
  \begin{split}
  X = \{x_{1,1}, x_{1,2}\} \times \{x_{2,1}, x_{2,2}, x_{2,3}\} = \big\{ &(x_{1,1}, x_{2,1}), (x_{1,1}, x_{2,2}), (x_{1,1}, x_{2,3}), \\
  & (x_{1,2}, x_{2,1}), (x_{1,2}, x_{2,2}), (x_{1,2}, x_{2,3}) \big\}
  \end{split}
\end{equation*}
in $\R^2$ and the prior mean $m \equiv 0$.

\begin{restatable}[Estimation of multiple lengthscales]{theorem}{multiplelengthscales} \label{thm:multiple-lengthscales}
  Consider the product kernel in~\eqref{eq:product-kernel} and suppose that the stationary kernels $K_1, \ldots, K_d$ on $\R$ satisfy \Cref{assumption:sobolev-kernel}.
  If $X$ has the product form~\eqref{eq:product-covariates} with $n_i \geq 1$ for each $i = 1, \ldots, d$ and the data $Y$ are $m$-constant along dimension $p \in \{1, \ldots, d\}$, then
  \begin{equation*}
    \lim_{ \lambda_p \to \infty} \ell( \lambda_1, \ldots, \lambda_d \mid Y) = -\infty \quad \text{ and } \quad \lambda_{p,\ML} = \infty,
  \end{equation*}
where $\lambda_{p,\ML}$ is the maximum likelihood estimate of the $p$th lengthscale parameter in~\eqref{eq:mle-lengthscale-p}.
\end{restatable}
\begin{proof}
See \Cref{sec:multiple-lambda-proof}. The product form of the kernel and the covariates allow one to write the full covariance matrix $K_\theta(X, X)$ as a Kronecker product of $K_{i,\lambda}(X_i, X_i)$. One may then utilise the properties of Kronecker products and subsequently follow the proof of \Cref{thm:main-theorem}.
\end{proof}

Note that the constants $C_1$, $C_2$ and $\alpha$ in \Cref{assumption:sobolev-kernel} may differ from one constituent kernel $K_i$ to another.
\Cref{thm:multiple-lengthscales} provides some theoretical justification for the use of maximum likelihood estimation of lengthscales as an \emph{automatic relevance determination} method~\citep[Section~5.1]{RasmussenWilliams2006}. 
When the data are independent of the $p$th input dimension, the lengthscale for this dimension is set to infinite and the dimension is effectively ignored. 

\subsection{Infinitely Smooth Stationary Kernels}

Commonly used infinitely smooth stationary kernels, such as the Gaussian and the inverse quadratic (or Cauchy) defined by 
\begin{equation} \label{eq:smooth-kernels}
  \Phi(z) = \exp( - \norm[0]{z}^2 ) \quad \text{ and } \quad \Phi(z) = \frac{1}{1 + \norm[0]{z}^2},
\end{equation}
respectively, do not satisfy \Cref{assumption:sobolev-kernel} because their Fourier transforms decay (at least) exponentially.
The exponential decay of their Fourier transforms implies that these kernels are analytic.
The purpose of \Cref{assumption:sobolev-kernel} is to guarantee that constant functions are contained in the RKHS of $K$ on a bounded set, a result which in turn can be exploited to prove that the data-fit term $Y_m^\T K_\lambda(X, X)^{-1} Y_m$ is a bounded function of $\lambda$ whenever the data are $m$-constant (see \Cref{lemma:data-fit-bounded}).
However, it is known that the RKHSs of analytic stationary kernels, such as those in~\eqref{eq:smooth-kernels}, do not contain constant functions~\citep{Steinwart2006, SunZhou2008, Minh2010, DetteZhigljavsky2021}.
But this does have to mean that the data-fit term explodes as $\lambda \to \infty$.

Increasing the lengthscale parameter is equivalent to \emph{coalescence} of the points to the origin. 
That is, using the kernel $K_\lambda$ and points $X$ is equivalent to using the kernel $K$ and the scaled points $X_\lambda = \{x_1 / \lambda, \ldots, x_n / \lambda\}$, each of which tends to the origin as $\lambda \to \infty$.
Suppose for simplicity that $d=1$.
When the points coalesce, one's data effectively comprises the value at the origin of the data-generating function and its successive derivatives up to order $n-1$.
We refer to Section~11 in \citet{Schaback2008} and Chapter~5 in \citet{Oettershagen2017} for more discussion and some results regarding this phenomenon. The computations in Section~2 of \citet{DetteZhigljavsky2021} are also relevant.
If the data are $m$-constant such that $Y_m = (c, \ldots, c)$ and the kernel $K(x, y) = \Phi(x -y)$ is sufficiently smooth, this reasoning suggests the conjecture that
\begin{equation} \label{eq:smooth-conjecture}
  \lim_{ \lambda \to \infty } Y_m^\T K_\lambda(X, X)^{-1} Y_m = \lim_{ \lambda \to \infty } Y_m^\T K(X_\lambda, X_\lambda)^{-1} Y_m = D_0^\T W^{-1} D_0
\end{equation}
where $D_0 = (c, 0, \ldots, 0) \in \R^n$ and the \emph{Wronskian} $W \in \R^{n \times n}$ has the elements
\begin{equation} \label{eq:wronskian}
  (W)_{i+1,j+1} = \frac{\partial^{i+j}}{\partial v^i \partial w^j} K(v, w) \Bigl|_{\substack{v = 0 \\ w = 0}}
\end{equation}
for $i, j = 0, \ldots, n-1$.
The conjectured limit $D_0^\T W^{-1} D_0$ is the data-fit term in~\eqref{eq:ell-general} for the information functionals defined as $L_i f = f^{(i-1)}(0)$ for $i=1,\ldots,n$ and general $m$-constant data in~\eqref{eq:m-constant-general}.
A proof of~\eqref{eq:smooth-conjecture} is the main ingredient in the proof of the following theorem, which partially generalises \Cref{thm:main-theorem} for infinitely differentiable kernels when $d = 1$.

\begin{restatable}{theorem}{smooththeorem} \label{thm:smooth-kernels}
  Let $d = 1$ and $K(x, y) = \Phi(x - y)$.
  Suppose that (i) the function $\Phi \colon \R \to \R$ is integrable and infinitely differentiable in a neighbourhood of the origin; (ii) $\Phi^{(k)}(0) = 0$ for every odd $k$; and (iii) the Fourier transform of $\Phi$ is everywhere positive.
  If $n \geq 1$ and the data $Y$ are $m$-constant, then
  \begin{equation*}
    \lim_{ \lambda \to \infty} \ell(\lambda \mid Y) = -\infty \quad \text{ and } \quad \lambda_\ML = \infty.
  \end{equation*}
\end{restatable}
\begin{proof}
See \Cref{sec:smooth-proof}. The proof uses Equation~(32) in \citet{BarthelmeUsevich2021}.
\end{proof}

\section{Conclusion and Implications} \label{sec:conclusion}

In this article we have proved that Gaussian process regression with noiseless data and a stationary covariance kernel is ill-posed if the lengthscale parameter of the kernel is estimated using maximum likelihood: When the data differ from the prior mean by a constant mean shift, the maximum likelihood estimate of the lengthscale parameter is infinite (\Cref{thm:main-theorem}) and the conditional Gaussian process is degenerate (\Cref{thm:flat-limit}).
As shown in \Cref{sec:what-does-not-help}, these conclusions remain valid under more general parametrisations and also applies to leave-one-out cross-validation.

\subsection{Practical Implications}

In a way, our results imply a practical simplification.
If the data are $m$-constant (which, when the prior mean $m$ is known and fixed, is trivial to check), there is no need for numerical optimisation of the log-likelihood function as one can use \Cref{thm:main-theorem} to set $\lambda_\ML = \infty$ and \Cref{thm:flat-limit} to compute the conditional mean and covariance.
However, degeneracy of the resulting conditional process implies that there is no predictive uncertainty and the conditional process is therefore useless as a tool for uncertainty quantification.
Except for switching to a non-stationary kernel, we do not know of a good approach to fix this, and it may be that some non-stationary kernels are equally problematic and induce similar behaviour when some of their parameters are estimated.

A numerical issue that is encountered when the data are close to being $m$-constant is that of ill-conditioning of the covariance matrix.
As $\lambda \to \infty$, the condition number of the covariance matrix $K_\lambda(X, X)$ increases with a rate related to the smoothness of the kernel.
This means that one cannot compute the modified log-likelihood function for large values of $\lambda$.
In practice one therefore has to either introduce a regularisation parameter to upper bound the condition number as a function of $\lambda$ or select a finite maximal lengthscale $\lambda_\textup{max} > 0$ for which $\ell( \lambda \mid Y)$ can be reliably computed and find the maximum likelihood estimate in $(0, \lambda_\textup{max}]$.
  When the data are $m$-constant, restricting the feasible set for $\lambda$ to $(0, \lambda_\textup{max}]$ is likely to result in $\lambda_\ML = \lambda_\textup{max}$, so that the user effectively selects an arbitrary (though probably fairly large) lengthscale in this case.
    The dependence of predictions in $\lambda_\textup{max}$ may or may not be problematic depending on the context.

    A practical recommendation borne out by our results is that all Gaussian process implementations which use maximum likelihood but not regularisation should check if the data are $m$-constant.
    If the check indicates that the data are $m$-constant, an implementation should either (a) forgo lengthscale estimation and output a degenerate conditional process or (b) throw an error and inform the user of the problem.
    If the approach~(a) is chosen it should be made clear to the user that the output is degenerate as this may have important ramifications in the applied context.
    A more general research programme suggests itself:
      \begin{enumerate}
      \item[(i)] To characterise, for each estimator of the kernel parameters $\theta$, the problematic data sets $Y$ which cause Gaussian process regression or interpolation to be ill-posed.
      \item[(ii)] To hard-code Gaussian process software to throw an error (or at least a warning) when such data are encountered.
      \end{enumerate}

\subsection{Theoretical Implications}

When analysing the convergence of Gaussian process regression as $n \to \infty$, it is typically assumed that the covariance kernel is fixed.
To the best of our knowledge, in the deterministic interpolation regime only \citet{Teckentrup2020} and \citet{Wynne2021} allow the kernel parameters other than scaling parameter $\sigma$ from \Cref{sec:sigma-lambda} (which does not affect the conditional mean) to vary.
Their results are generic in that no specific parameter estimation method is considered and the parameter estimates are simply assumed to remain within certain sets.
In \citet{Wynne2021} only a smoothness parameter, such as the parameter $\nu$ of Matérn kernels~\eqref{eq:matern}, is allowed to vary.
\citet[Theorem~3.5]{Teckentrup2020} considers kernels which satisfy~\Cref{assumption:sobolev-kernel} and proves that the conditional mean in~\eqref{eq:conditional-mean} tends to the true data-generating function $f$ such that $Y = f(X)$ if (i) this function has certain smoothness and (ii) there is a compact set which contains the estimate of $\lambda$ for every $n$.
As we have seen in this article, the second assumption fails if $f$ happens to be a mean shift of the prior mean (i.e., $f = m + c$ for some $c \in \R$) and $\lambda$ is estimated using maximum likelihood.
This demonstrates that unless one imposes an artificial upper bound on the parameter estimates, smoothness assumptions alone are not sufficient for comprehensive convergence analysis of Gaussian process regression.

\subsection{On Estimation of Other Parameters} \label{sec:other-parameters}

We conclude by pointing out that our ill-posedness results are specific to lengthscale estimation and should not be expected to extend to estimation of other kernel parameters.
Two examples serve to illustrate this.
First, consider the scale parametrisation $K_\sigma(x, y) = \sigma^2 K(x, y)$ for a scale parameter $\sigma > 0$.
From~\eqref{eq:ell-function} it is straightforward to compute that the maximum likelihood estimate of $\sigma$ is available in closed form:
\begin{equation*}
  \sigma_\ML = \sqrt{ \frac{Y_m^\T K(X, X)^{-1} Y_m}{n} }.
\end{equation*}
Here only the data $Y = m(X)$ yield a problematic parameter estimate $\sigma_\ML = 0$ that results in degenerate predictive distributions.
Consider then estimation of the smoothness parameter $\nu > 0$ of a Matérn kernel $K_\nu$ in~\eqref{eq:matern}. 
The presence of the coefficient $2^{1-\nu}/\Gamma(\nu)$ ensures that $K_\nu(x, x) = \sigma^2$ for every $\nu > 0$, which has two implications:
\begin{itemize}
\item As is well known, $K_\nu(x, y)$ tends to the Gaussian kernel $\sigma^2 \exp(-\norm[0]{x-y}^2\!/(2\lambda^2))$ as $\nu \to \infty$ for all $x, y \in \R^d$.
\item As $\nu \to 0$, $K_\nu(x, y) \to \sigma^2$ if $x = y$ and $K_\nu(x, y) \to 0$ if $x \neq y$. The latter of these claims follows from the facts that $\mathcal{K}_0(z)$, the Bessel function of the second kind of zeroth order, is well-defined if $z \neq 0$ and $\Gamma(\nu) \to \infty$ as $\nu \to 0$.
\end{itemize}
This shows that both potentially problematic limiting cases, $\nu \to 0$ and $\nu \to \infty$, yield valid positive-definite kernels.
Consequently, degenerate predictive distributions can never arise from estimation of the Matérn smoothness parameter.

\section{Proofs} \label{sec:proofs}

This section contains proofs for the results in \Cref{sec:main,sec:what-does-not-help,sec:regularisation,sec:generalisation}.

\subsection{Interpolation in Reproducing Kernel Hilbert Spaces} \label{sec:rkhs}

Let $\Omega$ be an arbitrary set and $K \colon \Omega \times \Omega \to \R$ a positive-definite kernel, which means that
\begin{equation} \label{eq:posdef-condition}
  \sum_{n=1}^n \sum_{j=1}^n a_i a_j K(x_i, x_j) > 0
\end{equation}
for any $n \in \N$, any non-zero vector $a = (a_1, \ldots, a_n)$, and any distinct points $x_i \in \Omega$.
The kernel is positive-semidefinite if the inequality in~\eqref{eq:posdef-condition} is not required to be strict.
Then $K$ induces a unique \emph{reproducing kernel Hilbert space} (RKHS), $H(K, \Omega)$.
This is a Hilbert space consisting of real-valued functions defined on $\Omega$ and is equipped with an inner product $\inprod{\cdot}{\cdot}_{H(K,\Omega)}$ and the corresponding norm $\norm[0]{\cdot}_{H(K,\Omega)}$.
The kernel translate $K(\cdot, x)$ is an element of $H(K, \Omega)$ for every $x \in \Omega$ and the kernel has the \emph{reproducing property}
\begin{equation*}
  \inprod{K(\cdot, x)}{f}_{H(K, \Omega)} = f(x) \quad \text{ for all } \quad x \in \Omega \: \text{ and } \: f \in H(K, \Omega).
\end{equation*}
It is usually not straightforward to determine whether or not a given function is an element of $H(K, \Omega)$.
However, the RKHS of a kernel which satisfies \Cref{assumption:sobolev-kernel} on the rate of decay of its Fourier transform is a Sobolev space; see \Cref{sec:sobolev-spaces}.
For more information on RKHSs we refer the reader to \citet{Berlinet2004} and Chapters~10 and~16 in \citet{Wendland2005}.

We are interested in optimal interpolation in an RKHS.
Let $f \colon \Omega \to \R$ be any function (i.e., not necessarily an element of the RKHS) that is to be interpolated at a set of distinct points $X = \{x_i\}_{i=1}^n \subset \Omega$.
The \emph{kernel interpolant} $s_{f,X}$ is the unique minimum norm interpolant to~$f$ at these points:
\begin{equation} \label{eq:minimum-norm-interpolation}
  s_{f, X} = \argmin_{ s \in H(K, \Omega) } \Set[\big]{ \norm[0]{s}_{H(K, \Omega)} }{ s(x_i) = f(x_i) \text{ for every } i = 1,\ldots, n}.
\end{equation}
The kernel interpolant has the explicit linear-algebraic form
\begin{equation} \label{eq:conditional-mean-s-proofs}
  s_{f, X}(x) = K(x, X)^\T K(X, X)^{-1} f(X),
\end{equation}
which equals the conditional mean in~\eqref{eq:conditional-mean} when $m \equiv 0$.
This is the famous equivalence between Gaussian process interpolation and optimal interpolation in an RKHS whose origins can be traced back at least to the work of \citet{KimeldorfWahba1970}.
From~\eqref{eq:conditional-mean-s-proofs} it is straightforward to compute that~\citep[e.g.,][Section~5.1]{Fasshauer2011}
\begin{equation} \label{eq:s-norm-explicit}
  \norm[0]{s_{f,X}}_{H(K, \Omega)}^2 = f(X)^\T K(X, X)^{-1} f(X),
\end{equation}
which equals the data-fit term in~\eqref{eq:ell-function} for $m \equiv 0$.
Note that a particular implication of~\eqref{eq:minimum-norm-interpolation} and~\eqref{eq:s-norm-explicit} is that $f(X)^\T K(X, X)^{-1} f(X) \leq \norm[0]{f}_{H(K, \Omega)}^2$ if $f \in H(K, \Omega)$.
How these properties of $s_{f, X}$ follow is explained in more detail in the proof of \Cref{prop:interpolant-norm-bound} concerning interpolation based on general linear data.
For the conditional variance we use the notation 
\begin{equation} \label{eq:conditional-variance-proofs}
  P_X(x)^2 = K(x, x) - K(x, X)^\T K(X, X)^{-1} K(X, x),
\end{equation}
which makes the dependency on the points $X$ explicit.
Now, for every $x \in \Omega$ it holds that~\citep[e.g.,][Theorem~11.4]{Wendland2005}
\begin{equation} \label{eq:RKHS-error-estimate}
  \abs[0]{ f(x) - s_{f,X}(x) } \leq \norm[0]{f}_{H(K,\Omega)} P_X(x)
\end{equation}
if $f \in H(K, \Omega)$, so that the conditional standard deviation controls the interpolation error.

\subsection{On Notation} \label{sec:on-notation}

The proofs require notation that is more expressive than what we have used elsewhere in this article.
Therefore the conditional variance in~\eqref{eq:conditional-variance} equals the conditional variance~\eqref{eq:conditional-variance-proofs} whose dependency on the covariates has been made explicit.
Similarly, the function $s$ in~\eqref{eq:conditional-mean-s} equals the kernel interpolant $s_{f,X}$ in~\eqref{eq:conditional-mean-s-proofs} for any function $f$ such that $f(X) = Y_m$.
It is often necessary or useful to indicate that various quantities depend on the kernel parameters (either~$\theta$ or~$\lambda$).
We use subscripts for this purpose.
Subscripts are also used as shorthands for point sets formed by removing some elements of $X = \{x_i\}_{i=1}^n$ in the following way: $X_k = \{x_i\}_{i=1}^k$ and $X_{n,k} = X \setminus \{x_k\}$.
Analogous notation is used for the conditional standard deviation and mean and the kernel interpolant constructed at these point sets, so that
\begin{equation*}
  P_{\theta, k} = P_{\theta,X_k} = P_{\theta, \{x_1, \ldots, x_k\}} \quad \text{ and } \quad P_{\theta, n, k} = P_{\theta, X_{n,k}} = P_{\theta, X \setminus \{x_k\}}
\end{equation*}
and
\begin{equation*}
  s_{\theta, k} = s_{\theta, f, X_k} = s_{\theta, f, \{x_1, \ldots, x_k\}} \quad \text{ and } \quad s_{\theta, n, k} = s_{\theta, f, X_{n,k}} = s_{\theta, f, X \setminus \{x_k\}}
\end{equation*}
for $k \leq n$ and $f$ such that $f(X) = Y_m$.
These notational conventions are reintroduced preceding their use in the proofs.

\subsection{Sobolev Spaces} \label{sec:sobolev-spaces}

For $\alpha \geq 0$, the Sobolev space $W_2^\alpha(\R^d)$ consists of square-integrable functions $f \colon \R^d \to \R$ such that
\begin{equation} \label{eq:Sobolev-norm}
  \norm[0]{f}_{W_2^\alpha(\R^d)}^2 = \int_{\R^d} ( 1 + \norm[0]{\xi}^2 )^\alpha \abs[0]{ \widehat{f}(\xi) }^2 \dif \xi < \infty.
\end{equation}
On $\Omega \subset \R^d$, the Sobolev space $W_2^\alpha(\Omega)$ consists of those $f \colon \Omega \to \R$ which admit an extension $f_e \in W_2^\alpha(\R^d)$ such that $f_e|_\Omega = f$.
The norm of $W_2^\alpha(\Omega)$ is
\begin{equation} \label{eq:sobolev-restriction}
  \norm[0]{f}_{W_2^\alpha(\Omega)} = \inf_{ f_e \in W_2^\alpha(\R^d) } \Set[\big]{ \norm[0]{f_e}_{W_2^\alpha(\R^d)}}{ f_e|_\Omega = f}.
\end{equation}
If $\alpha \in \N_0$, $W_2^\alpha(\R^d)$ consists of functions whose weak derivatives up to order $\alpha$ exist and are square-integrable.
It is a standard result~\citep[e.g.,][Corollary~10.13]{Wendland2005} that for a kernel $K$ which satisfies \Cref{assumption:sobolev-kernel} the RKHS $H(K, \R^d)$ is norm-equivalent to $W_2^\alpha(\R^d)$.
This is to say that $H(K, \R^d)$ and $W_2^\alpha(\R^d)$ are equal as sets and that there are positive constants $C_1$ and $C_2$ such that
\begin{equation*}
  C_1 \norm[0]{f}_{W_2^\alpha(\R^d)} \leq \norm[0]{f}_{H(K, \R^d)} \leq C_2 \norm[0]{f}_{W_2^\alpha(\R^d)}
\end{equation*}
for every $f \in H(K, \R^d)$.
An analogous result carries over to $H(K, \Omega)$, which is related to $H(K, \R^d)$ in the same way as $W_2^\alpha(\Omega)$ is to $W_2^\alpha(\R^d)$ via~\eqref{eq:sobolev-restriction}.
We use the following two facts in the proof of \Cref{thm:main-theorem}:
\begin{itemize}
\item Let $B$ be any open ball centered at the origin. Then constant functions are contained in $W_2^\alpha(\Omega)$ for any $\alpha \geq 0$ because one can construct a bump function which is constant in $B$ and whose Fourier transform decays with a super-algebraic rate.
\item
  The function $\Phi$, which defines $K$ in \Cref{assumption:sobolev-kernel}, is Hölder continuous with the exponent \sloppy{${\beta(\alpha) = \min\{1, \alpha - d/2 \} > 0}$} on any sufficiently regular domain $\Omega$ (e.g., an open ball).
  That is, there is a positive constant $C$ such that
  \begin{equation} \label{eq:holder-condition}
    \abs[0]{ \Phi(0) - \Phi(x) } \leq C \norm[0]{x}^{\beta(\alpha)}
  \end{equation}
  for any $x \in \Omega$.
  This assertion is a consequence of the classical inclusion relation between Sobolev and Hölder spaces~\citep[e.g.,][Remark~2 on p.\@~206]{Triebel1978}.
  That $\Phi$ is an element of $W_2^\alpha(\R^d)$ is easy to verify using~\eqref{eq:Phi-bounds} and~\eqref{eq:Sobolev-norm}.
\end{itemize}

\subsection{Proof of \Cref{thm:main-theorem}} \label{sec:proof-main}

We split the proof of \Cref{thm:main-theorem}, which is repeated below, in two.

\maintheorem*

\Cref{eq:mle-infinite} is a rather straightforward consequence of~\eqref{eq:s-norm-explicit} when one interprets $\lambda$ as a scaling of the covariate set instead of a kernel parameter, while proving \Cref{eq:mle-finite} requires some more work, including upper and lower bounds for the conditional variance.
Most of the proof of \Cref{eq:mle-infinite} is contained in the following lemmas, which will be used again in \Cref{sec:proofs-what-does-not-help}.

  \begin{lemma}[Continuity of the data-fit and model complexity] \label{lemma:continuity}
  Suppose that $\Phi$ is continuous and $n \geq 1$.
  Then the functions
  \begin{equation*}
    f_\textup{df}(\lambda) = Y_m^\T K_\lambda(X, X)^{-1} Y_m \quad \text{ and } \quad f_\textup{mc}(\lambda) = \log \det K_\lambda(X, X)
  \end{equation*}
  are well-defined and continuous on $(0, \infty)$.
\end{lemma}
\begin{proof}
  Define $X_\lambda = \{x_i / \lambda\}_{i=1}^n$ and observe that $K_\lambda(X, X) = K(X_\lambda, X_\lambda)$.
  Since $K$ is a positive-definite kernel and the covariates $X$ are distinct, this shows that $K_\lambda(X, X)$ is positive-definite and hence non-singular for every $\lambda > 0$.
  In particular, $\det K_\lambda(X, X) > 0$ for every $\lambda > 0$ by positive-definiteness.
  Therefore the functions $f_\textup{df}$ and $f_\textup{mc}$ are well-defined.
  Because $K(x, y) = \Phi(x - y)$ is continuous, each element of $K_\lambda(X, X)$ is a continuous function of $\lambda$.
  From the definition of the determinant it immediately follows that $f_\textup{mc}$ is continuous.
  The continuity of $f_\textup{df}$ is then a consequence of, for example, Cramer's rule and the positivity of $\det K_\lambda(X, X)$.
\end{proof}

\begin{lemma}[Boundedness of the data-fit term] \label{lemma:data-fit-bounded}
  Suppose that $K$ satisfies \Cref{assumption:sobolev-kernel} and $n \geq 1$.
  If the data $Y$ are $m$-constant, then there is a constant $a > 0$ such that
  \begin{equation*}
    \sup_{ \lambda > 0 } Y_m^\T K_\lambda(X, X)^{-1} Y_m \leq a.
  \end{equation*}
\end{lemma}
\begin{proof}
  Because the data $Y$ are $m$-constant, we can write $Y_m = Y - m(X) = f(X)$ for some constant function $f$.
  Define $X_\lambda = \{x_i / \lambda\}_{i=1}^n$ and observe that, since $f(X_\lambda)$ does not depend on~$\lambda$,
  \begin{equation*}
    Y_m^\T K_\lambda(X, X)^{-1} Y_m = f(X_\lambda)^\T K(X_\lambda, X_\lambda)^{-1} f(X_\lambda).
  \end{equation*}
  The RKHS $H(K, B)$ contains constant functions if $B$ is any open ball centered at the origin by the results in \Cref{sec:sobolev-spaces}.
  Because $X$ contains a finite number of points, we can trivially select $B$ such that $X \subset B$.
  Then the set $X_\lambda$ is also contained in $B$ whenever $\lambda \geq 1$.
  Therefore $f \in H(K, B)$ and it follows from~\eqref{eq:minimum-norm-interpolation} and~\eqref{eq:s-norm-explicit} that
  \begin{equation} \label{eq:min-norm-interpolation-in-proof}
    f(X_\lambda)^\T K(X_\lambda, X_\lambda)^{-1} f(X_\lambda) \leq \norm[0]{f}_{H(K,B)}^2 
  \end{equation}
  if $\lambda \geq 1$.
  Because $K(x, y) = \Phi(x - y)$ for $\Phi$ which is continuous and integrable on $\R^d$,
  \begin{equation*}
    K\bigg( \frac{x_i}{\lambda}, \frac{x_i}{\lambda} \bigg) = \Phi(0) > 0 \quad \text{ and } \quad \lim_{\lambda \to 0} K\bigg( \frac{x_i}{\lambda}, \frac{x_j}{\lambda} \bigg) = \lim_{\lambda \to 0} \Phi\bigg( \frac{x_i - x_j}{\lambda} \bigg) = 0
  \end{equation*}
  for all $i \neq j$.
  That is, $K_\lambda(X, X) = K(X_\lambda, X_\lambda)$ tends to a non-zero diagonal matrix as $\lambda \to 0$.
  Thus by \Cref{lemma:continuity}, which guarantees the $\lambda$-continuity of the data-fit term,
  \begin{equation} \label{eq:data-fit-lim-lambda=0}
    \lim_{\lambda \to 0} Y_m^\T K_\lambda(X, X)^{-1} Y_m = \Phi(0)^{-1} \norm[0]{Y_m}^2 < \infty.
  \end{equation}
  From~\eqref{eq:min-norm-interpolation-in-proof} and~\eqref{eq:data-fit-lim-lambda=0} and \Cref{lemma:continuity} we conclude that
there is a constant $a > 0$ such that
  \begin{equation*}
    Y_m^\T K_\lambda(X, X)^{-1} Y_m = f(X_\lambda)^\T K(X_\lambda, X_\lambda)^{-1} f(X_\lambda) \leq a
  \end{equation*}
  for every $\lambda > 0$.
\end{proof}

\noindent\textbf{Proof of \Cref{eq:mle-infinite}}
  By \Cref{lemma:data-fit-bounded}, there is $a > 0$ such that
  \begin{equation} \label{eq:mle-1st-term-bound}
    0 \leq Y_m^\T K_\lambda(X, X)^{-1} Y_m \leq a
  \end{equation}
  for all $\lambda > 0$.
  Stationarity and continuity of $\Phi$ imply that $K_\lambda(X, X)$ converges to the identity matrix times $\Phi(0)$ as $\lambda \to 0$ and to the singular matrix of $\Phi(0)$'s as $\lambda \to \infty$.
  Thus it follows from \Cref{lemma:continuity} that
  \begin{equation} \label{eq:mle-2nd-term-limit}
    \log \det K_\lambda(X, X) \to -\infty
  \end{equation}
  if and only if $\lambda \to \infty$.
  By combining~\eqref{eq:mle-1st-term-bound} and~\eqref{eq:mle-2nd-term-limit} we conclude that 
  \begin{equation*}
    \ell( \lambda \mid Y ) = Y_m^\T K_\lambda(X, X)^{-1} Y_m + \log \det K_\lambda(X, X) \to -\infty
  \end{equation*}
  if and only if $\lambda \to \infty$.
  Therefore $\lambda_\ML = \argmin_{ \lambda > 0} \ell( \lambda \mid Y ) = \infty$.
  \qed

Three auxiliary results are needed to prove \Cref{eq:mle-finite}.
The first of these---or its variants---is well known in scattered data approximation literature~\citep[e.g.,][]{Schaback1995}.
The version that we need here is contained in the proof of Theorem~4.4 in \citet{Karvonen2020}.

\begin{proposition} \label{prop:Sobolev-P-lower-bound}
  Suppose that $K$ satisfies \Cref{assumption:sobolev-kernel} and let $X = \{x_i\}_{i=1}^n$ be any set of distinct points in $\R^d$.
  Define $\norm[0]{x - X} = \min_{i=1,\ldots,n} \norm[0]{x - x_i}$.
  Then there is a positive constant~$C$, which does not depend on $x$ or $X$, such that
  \begin{equation*}
    P_X(x)^2 \geq C \norm[0]{x - X}^{2\alpha - d}
  \end{equation*}
  for any $x \in \R^d$ for which $\norm[0]{x - X} \leq 1$, where $\alpha > d/2$ is the constant in \Cref{assumption:sobolev-kernel}.
\end{proposition}

The second auxiliary result, which is standard and essentially Exercise~8.66 in \citet{Iske2018}, gives a rough upper bound on the conditional variance under \Cref{assumption:sobolev-kernel}.

\begin{proposition} \label{prop:Sobolev-P-upper-bound}
  Suppose that $K$ satisfies \Cref{assumption:sobolev-kernel} and let $X = \{x_i\}_{i=1}^n$ be any set of distinct points in $\R^d$.
  Let $\norm[0]{x - X} = \min_{i=1,\ldots,n} \norm[0]{x - x_i}$ and $\beta(\alpha) = \min\{1, \alpha - d/2\} > 0$, where $\alpha > d/2$ is the constant in \Cref{assumption:sobolev-kernel}.
  Then there is a positive constant $C$, which does not depend on $x$ or $X$, such that
  \begin{equation*}
    P_X(x)^2 \leq C \norm[0]{x - X}^{\beta(\alpha)}
  \end{equation*}
  for any $x \in \R^d$ for which $\norm[0]{x - X}$ is sufficiently small.
\end{proposition}
\begin{proof}
  Let $x^* \in X$ be such that $\norm[0]{x - x^*} = \norm[0]{x - X}$.
  Because the standard deviation is a non-decreasing function in that $P_X(x) \leq P_{X'}(x)$ for any $x$ if $X' \subset X$~\citep[e.g.,][Theorem~16.11]{Wendland2005}, we have
  \begin{equation} \label{eq:P-decreasing-function}
    P_X(x) \leq P_{\{x^*\}}(x).
  \end{equation}
  Using the stationarity assumption and~\eqref{eq:conditional-variance-proofs} we write
  \begin{equation*}
    \begin{split}
    P_{\{x^*\}}(x)^2 = K(x, x) - \frac{K(x, x^*)^2}{K(x^*, x^*)} &= \Phi(0) - \frac{\Phi( x - x^*)^2}{\Phi(0)} \\
    &= \frac{1}{\Phi(0)} \big( \Phi(0) + \Phi( x - x^*) \big) \big( \Phi(0) - \Phi(x - x^*) \big).
    \end{split}
  \end{equation*}
  Applying the Hölder condition~\eqref{eq:holder-condition} to $\abs[0]{\Phi(0) - \Phi(x - x^*)}$ and using $\Phi(0) \geq \Phi(x - x^*)$, which follows from the positive-definiteness of $K(x, y) = \Phi(x - y)$, yields the estimate
  \begin{equation*}
    P_{\{x^*\}}(x)^2 \leq C \norm[0]{x - x^*}^{\beta(\alpha)} = C \norm[0]{x - X}^{\beta(\alpha)}
  \end{equation*}
  for a certain positive constant $C$ which depends only on $\Phi$.
  Using this bound in~\eqref{eq:P-decreasing-function} concludes the proof.
\end{proof}

Our third auxiliary result is an expression for the modified log-likelihood function.
Although this expression has appeared in the literature~\citep[e.g.,][Section~4.2.2]{XuStein2017}, we have not encountered its proof and therefore provide one based entirely on linear algebra, block matrix inversion, and determinantal identities.
The expressions for the individual terms of the modified log-likelihood function are relatively well known.
For the data-fit term, see \citet[Theorem~6]{SchabackWerner2006} or \citet[Bemerkung 3.1.4]{Muller2008}.
We also point the reader to Section~3 in \citet{Scheuerer2011}.
The expression for the model complexity term appear in literature on determinantal point processes~\citep[e.g.,][Section~2.4]{BardenetHardy2020}.
For the purposes of this proposition and the proof of \Cref{eq:mle-finite} we use the notation
\begin{equation*}
  P_{\theta,k}(x) = P_{\theta, \{x_1, \ldots, x_k\}}(x) \quad \text{ and } \quad s_{\theta,k}(x) = s_{\theta,f,\{x_1, \ldots, x_k\}}(x)
\end{equation*}
for $k \leq n$.

\begin{proposition} \label{prop:ell-representation}
  If the points $X = \{x_i\}_{i=1}^n$ are distinct, then
  \begin{equation} \label{eq:data-fit-representation} 
    Y_m^\T K_\theta(X, X)^{-1} Y_m = \sum_{k=0}^{n-1} \bigg( \frac{y_{k+1} - m(x_{k+1}) - s_{\theta,k}(x_{k+1})}{P_{\theta,k}(x_{k+1})} \bigg)^2
  \end{equation}
  and
  \begin{equation} \label{eq:model-complexity-representation}
    \det K_\theta(X, X) = \prod_{k=0}^{n-1} P_{\theta,k}(x_{k+1})^2,
  \end{equation}
  where $P_{\theta,0}(x_1)^2 = K_\theta(x_1, x_1)$ and $s_{\theta,0}(x_1) = 0$.
\end{proposition}
\begin{proof}
  Denote $y_{m,k} = y_k - m(x_k)$ and $Y_{m,k} = (y_{m,1}, \ldots, y_{m,k}) \in \R^k$ so that
  \begin{equation*}
    Y_m^\T K_\theta(X, X)^{-1} Y_m = 
    \begin{pmatrix}
      y_{m,n} \\
      Y_{m,n-1}
    \end{pmatrix}^\T
    \begin{pmatrix}
      K_\theta(x_n, x_n) & K_\theta(x_n, X_{n-1})^\T \\
      K_\theta(x_n, X_{n-1}) & K_\theta(X_{n-1}, X_{n-1}) 
    \end{pmatrix}^{-1}
    \begin{pmatrix}
      y_{m,n} \\
      Y_{m,n-1}
    \end{pmatrix}.
  \end{equation*}
  The block matrix inversion formula, a few lines of straightforward algebra, and~\eqref{eq:conditional-mean} and~\eqref{eq:conditional-variance} then yield
  \begin{equation*}
    Y_m^\T K_\theta(X, X)^{-1} Y_m = \bigg( \frac{y_{m,n} - s_{\theta,n-1}(x_n) }{P_{\theta,n-1}(x_n)} \bigg)^2 + Y_{m,n-1}^\T K_\theta(X_{n-1}, X_{n-1})^{-1} Y_{m,n-1},
  \end{equation*}
  iteration of which yields the form~\eqref{eq:data-fit-representation} for the data-fit term.
  The block determinant identity and the expression~\eqref{eq:conditional-variance} for the conditional variance yield
  \begin{equation*}
    \begin{split}
    \det K_\theta(X, X) ={}&
    \begin{pmatrix}
      K_\theta(x_n, x_n) & K_\theta(x_n, X_{n-1})^\T \\
      K_\theta(x_n, X_{n-1}) & K_\theta(X_{n-1}, X_{n-1}) 
    \end{pmatrix} \\
    ={}& \big[ K_\theta(x_n, x_n) - K_\theta(x_n, X_{n-1})^\T K_\theta(X_{n-1}, X_{n-1})^{-1} K_\theta(x_n, X_{n-1}) \big] \\
    &\quad \times \det K_\theta(X_{n-1}, X_{n-1}) \\
    ={}& P_{\theta,n-1}(x_n)^2 \det K_\theta(X_{n-1}, X_{n-1}),
    \end{split}
  \end{equation*}
  iteration of which produces~\eqref{eq:model-complexity-representation}.
\end{proof}

\begin{lemma} \label{lemma:model-complexity-rate}
  Suppose that $K$ satisfies \Cref{assumption:sobolev-kernel} and let $X = \{x_i\}_{i=1}^n$ be any set of distinct points in $\R^d$.
  Then there is a constant~$C$, which does not depend on $\lambda$, such that
  \begin{equation*}
    \log \det K_\lambda(X, X) \geq -(2\alpha - d) n \log \lambda + C
  \end{equation*}
  for every sufficiently large $\lambda > 0$, where $\alpha > d/2$ is the constant in \Cref{assumption:sobolev-kernel}.
\end{lemma}
\begin{proof}
  Let $X_k = \{x_1, \ldots, x_k\}$ and $X_{\lambda,k} = \lambda^{-1} X_k = \{\lambda^{-1} x_i \}_{i=1}^k$.
  Because
  \begin{equation*}
    P_{\lambda,k}(x_{k+1}) = P_{X_{\lambda,k}}( \lambda^{-1} x_{k+1} )
  \end{equation*}
  and
  \begin{equation*}
    \begin{split}
      \norm[0]{\lambda^{-1} x_{k+1} - X_{\lambda,k}} = \min_{i=1,\ldots,k} \norm[0]{ \lambda^{-1} (x_{k+1} - x_i)} &= \lambda^{-1} \min_{i=1,\ldots,k} \norm[0]{x_{k+1} - x_i} \\
      &= \lambda^{-1} \norm[0]{x_{k+1} - X_k},
      \end{split}
  \end{equation*}
  it follows from \Cref{prop:Sobolev-P-lower-bound}, when applied to the points $X_{\lambda,k}$, that, for a positive constant $C$ which does not depend on $k$ or $\lambda$,
  \begin{equation*}
    \begin{split}
      \log [ P_{\lambda,k}(x_{k+1})^2 ] &= \log [ P_{X_{\lambda,k}}( \lambda^{-1} x_{k+1})^2 ] \\
      &\geq - (2\alpha - d) \log \lambda + (2\alpha - d) \log \norm[0]{x_{k+1} - X_k} + \log C
      \end{split}
  \end{equation*}
  when $\lambda$ is large enough that $\norm[0]{\lambda^{-1} x_{k+1} - X_{\lambda,k}} \leq 1$.
  Equation~\eqref{eq:model-complexity-representation} implies that
  \begin{equation*}
      \log \det K_\lambda(X, X) = \sum_{k=0}^{n-1} \log [ P_{\lambda,k}(x_{k+1})^2 ],
  \end{equation*}
  which yields the claim.
\end{proof}

\noindent\textbf{Proof of \Cref{eq:mle-finite}}
  Under the assumption that the data are not $m$-constant we can freely order $\{x_1, \ldots, x_n\}$ such that $y_1 - m(x_1) \neq y_2 - m(x_2)$.
  Then for the second term of the data-fit in~\eqref{eq:data-fit-representation} we have
  \begin{equation} \label{eq:second-term-explicit}
    \frac{y_{2} - m(x_{2}) - s_{\lambda,1}(x_{2})}{P_{\lambda,1}(x_{2})} = \frac{y_{2} - m(x_2) - K_\lambda(x_2, x_1)K_\lambda(x_1, x_1)^{-1}(y_1 - m(x_1)) }{P_{\lambda,1}(x_{2})}.
  \end{equation}
  By the stationarity assumption, the numerator on the right-hand side of~\eqref{eq:second-term-explicit} is
  \begin{equation*}
    y_{2} - m(x_2) - \Phi\bigg( \frac{x_2 - x_1}{\lambda} \bigg)\Phi(0)^{-1} (y_1 - m(x_1)).
  \end{equation*}
  As $\lambda \to \infty$, the numerator therefore tends to $y_{2} - m(x_2) - (y_1 - m(x_1)) \neq 0$.
  Consequently, there is a positive constant $C_1$ independent of $\lambda$ such that
  \begin{equation*}
    \begin{split}
      \ell( \lambda \mid Y ) &= Y_m^\T K_\lambda(X, X)^{-1} Y_m + \log \det K_\lambda(X, X) \\
      &= \sum_{k=0}^{n-1} \bigg( \frac{y_{k+1} - m(x_{k+1}) - s_{\lambda,k}(x_{k+1})}{P_{\lambda,k}(x_{k+1})} \bigg)^2 + \log \det K_\lambda(X, X) \\
      &\geq \frac{C_1}{P_{\lambda,1}(x_{2})^2} + \log \det K_\lambda(X, X)
    \end{split}
  \end{equation*}
  for all sufficiently large $\lambda$, where we have discarded all other terms of the data-fit in~\eqref{eq:data-fit-representation} except the $k=1$ term.
  \Cref{prop:Sobolev-P-upper-bound} yields the upper bound
  \begin{equation} \label{eq:P12-term-upper-bound}
    P_{\lambda,1}(x_{2})^2 \leq C_2 \lambda^{-\beta(\alpha)}
  \end{equation}
  for $\beta(\alpha) > 0$, a certain positive constant $C_2$ independent of $\lambda$, and all sufficiently large $\lambda$.
  This bound and \Cref{lemma:model-complexity-rate} then give
  \begin{equation*}
    \ell( \lambda \mid Y ) \geq \frac{C_1}{P_{\lambda,1}(x_{2})^2} + \log \det K_\lambda(X, X) \geq C_1 C_2^{-1} \lambda^{\beta(\alpha)} - (2\alpha - d) n \log \lambda + C_3
  \end{equation*}
  when $\lambda$ is sufficiently large and where none of the constants depend on $\lambda$.
  Therefore
  \begin{equation*}
    \lim_{ \lambda \to \infty} \ell(\lambda \mid Y) \geq \lim_{\lambda \to \infty} \big( C_1 C_2^{-1} \lambda^{\beta(\alpha)} - (2\alpha - d) n \log \lambda \big) = \infty
  \end{equation*}
  since $\beta(\alpha) > 0$.
  This concludes the proof.
  \qed

We point out that the lower bound on $\log \det K_\lambda(X, X)$ in \Cref{lemma:model-complexity-rate} is of independent interest.
See Theorems~4.5 and~6.3 in \citet{BarthelmeUsevich2021} for other results on the behaviour of $\det K_\lambda(X, X)$ as $\lambda \to \infty$.

\subsection{Proof of \Cref{thm:flat-limit}} \label{sec:proofs-flat-limit}

\flatlimittheorem*
\begin{proof}
  Because the data are $m$-constant, there is a constant function $f \equiv c$ for some $c \in \R$ such that $Y_m = Y - m(X) = f(X)$.
  Recall from \Cref{sec:sobolev-spaces} that \Cref{assumption:sobolev-kernel} implies that $f$ is an element of $H(K, B)$ for an open ball $B$ that can be selected such that $X \subset B$.
  Let $X_\lambda = \{x_i / \lambda \}_{i=1}^n$ and write
  \begin{equation*}
    \mu_\lambda(x) = m(x) + s_\lambda(x) = m(x) + s_{f,X_\lambda}(\lambda^{-1} x).
  \end{equation*}
  The RKHS error estimate~\eqref{eq:RKHS-error-estimate} yields, for any $\lambda \geq 1$,
  \begin{equation} \label{eq:error-estimate-for-flat-limit}
    \abs[0]{ m(x) + c - \mu_\lambda(x) } = \abs[0]{ f(\lambda^{-1} x) - s_{f,X_\lambda}(\lambda^{-1}x) } \leq \norm[0]{f}_{H(K,B)} P_{X_\lambda}(\lambda^{-1} x).
  \end{equation}
  As in the proof of \Cref{eq:mle-finite}, \Cref{prop:Sobolev-P-upper-bound} yields the bound
  \begin{equation} \label{eq:flat-limit-P-Sobolev-bound}
    P_\lambda(x)^2 = P_{X_\lambda}( \lambda^{-1} x )^2 \leq C \lambda^{-\beta(\alpha)}
  \end{equation}
  for $\beta = \min\{1, \alpha - d/2\} > 0$ and a positive constant $C$ which does not depend on $\lambda$.
  The claim $\lim_{\lambda \to \infty} \mu_\lambda(x) = m(x) + c$ is proved by inserting this bound in~\eqref{eq:error-estimate-for-flat-limit}.
  The claim for the conditional covariance follows from~\eqref{eq:flat-limit-P-Sobolev-bound} and the Cauchy--Schwarz covariance inequality $P_\lambda(x, y)^2 \leq P_\lambda(x) P_\lambda(y)$.
\end{proof}

\begin{remark} \label{remark:flat-limit-generalisation}
  Note in the above proof it is sufficient that $P_\lambda(x)$ tends to zero as $\lambda \to \infty$. Unlike in the proof of \Cref{eq:mle-finite}, the rate with which this convergence occurs is of no importance. From the proof of \Cref{prop:Sobolev-P-upper-bound} it is easy to see that a sufficient condition for $\lim_{ \lambda \to \infty} P_\lambda(x) = 0$ is that $K(x, y) = \Phi(x - y)$ for a continuous $\Phi \colon \R^d \to \R$.
  Therefore \Cref{thm:flat-limit} holds if $K$ is a continuous stationary kernel such that constant functions are contained in $H(K, B)$ for some open ball $B$ centered at the origin.
  This generalisation can be extended also to \Cref{eq:mle-infinite} and \Cref{thm:cross-validation}.
  Further generalisations are discussed in \Cref{sec:generalisation}.
\end{remark}

\subsection{Proofs for \Cref{sec:what-does-not-help}} \label{sec:proofs-what-does-not-help}

Recall the notational conventions reviewed in \Cref{sec:on-notation}.
In addition, denote
\begin{equation*}
  X_{\lambda,n,k} = \lambda^{-1} X_{n,k} = \{ \lambda^{-1} x_i \}_{i \neq k}.
\end{equation*}

\crossvalidation*
\begin{proof}
  Since the data are $m$-constant, there is a constant function $f$ such that $Y_m = Y - m(X) = f(X)$.
  Recall from \Cref{sec:sobolev-spaces} that \Cref{assumption:sobolev-kernel} implies that $f$ is an element of $H(K, B)$ for an open ball $B$ that can be selected such that $X \subset B$.
  Observe that $\mu_{\lambda,n,k}(x) = m(x) + s_{\lambda,n,k}(x)$.
  \Cref{prop:Sobolev-P-upper-bound} and the RKHS error estimate~\eqref{eq:RKHS-error-estimate} yield, for any $\lambda \geq 1$ and a constant $C_1$ which does not depend on $\lambda$,
  \begin{equation*}
    \begin{split}
      \ell_\CV( \lambda \mid Y ) &= \sum_{k = 1}^n \bigg[ \bigg( \frac{ y_k - \mu_{\lambda,n,k}(x_k) }{P_{\lambda,n,k}(x_k)} \bigg)^2 + \log [ P_{\lambda,n,k}(x_k)^2 ] \bigg] \\
      &= \sum_{k = 1}^n \bigg[ \bigg( \frac{ f(\lambda^{-1} x_k) - s_{f,X_{\lambda,n,k}}(\lambda^{-1} x_k) }{P_{X_{\lambda,n,k}}(\lambda^{-1} x_k)} \bigg)^2 + \log [ P_{X_{\lambda,n,k}}(\lambda^{-1} x_k)^2 ] \bigg] \\
      &\leq n \norm[0]{f}_{H(K,B)}^2 - \beta(\alpha) n \log \lambda + C_1,
      \end{split}
  \end{equation*}
  which establishes that
  \begin{equation} \label{eq:proof-ell-cv-inf-limit}
    \lim_{ \lambda \to \infty} \ell_\CV( \lambda \mid Y) = -\infty
  \end{equation}
  since $\beta(\alpha) = \min\{1, \alpha - d/2\}> 0$.
  Because $K_\lambda(x, y) \to 0$ as $\lambda \to 0$ if $x \neq y$, it is easy to compute from~\eqref{eq:conditional-mean} and~\eqref{eq:conditional-variance} that
  \begin{equation*}
    \lim_{\lambda \to 0 } s_{\lambda,n,k}(x_k) = 0 \quad \text{ and} \quad \lim_{ \lambda \to 0 } P_{\lambda,n,k}(x_k) = \sqrt{\Phi(0)} > 0
  \end{equation*}
  for every $k$.
  Therefore
  \begin{equation*}
    \lim_{ \lambda \to 0} \ell_\CV( \lambda \mid Y ) =  c n \Phi(0)^{-1} + n \log \Phi(0) > -\infty,
  \end{equation*}
  so that it follows from Equation~\eqref{eq:proof-ell-cv-inf-limit} and the fact that $\ell_\CV(\lambda \mid Y)$ is continuous that \sloppy{${\lambda_\CV = \argmin_{ \lambda > 0} \ell_\CV(\lambda \mid Y) = \infty}$}.
\end{proof}

\begin{remark} \label{remark:cross-validation-proof-issue}
  Proving that $\lim_{ \lambda \to \infty} \ell_{\CV}(\lambda \mid Y) = \infty$ and $\lambda_{\CV} < \infty$ if the data are not $m$-constant is more challenging than the same for the maximum likelihood estimate.
  This is because in the proof of \Cref{eq:mle-finite} one could exploit a recursive form of the modified log-likelihood function that follows from \Cref{prop:ell-representation} and obtain a closed-form expression for the term in which the conditional mean and variance are based on one data point.
  If $n > 2$, it is more difficult to analyse any of the analogous terms in $\ell_\CV( \lambda \mid Y )$ because each of these terms is built out of conditional moments based on $n-1$ points.
\end{remark}

\parametrictheorem*
\begin{proof}
  The modified log-likelihood function is
  \begin{equation*}
    \ell( \lambda, \beta \mid Y ) = (Y - V(X) \beta)^\T K_\lambda(X, X)^{-1} (Y - V(X) \beta) + \log \det K_\lambda(X, X).
  \end{equation*}
  By \Cref{thm:main-theorem} and identification of $Y - V(X) \beta$ as the vector $Y_m$, $\ell( \lambda, \beta \mid Y)$ tends to negative infinity for a fixed $\beta \in \R^q$ if and only if $Y - V(X) \beta = (c, \ldots, c) \in \R^n$ for some $c \in \R$ and  $\lambda \to \infty$.
  Furthermore, because the data-fit term is non-negative and model complexity does not depend on $\beta$, $\ell( \lambda, \beta \mid Y)$ cannot tend to negative infinity if $\lambda$ is kept fixed.
  We conclude that $\lambda_\ML = \infty$ because by the assumption in~\eqref{eq:constant-data-parametric} there is $\beta^* \in \R^q$ for which $Y - V(X) \beta^* = (c, \ldots, c)$.
\end{proof}

\simutheorem*
\begin{proof}
  Recall that $\sigma_\ML$ and $\lambda_\ML$ are any minimisers of
  \begin{equation} \label{eq:ell-sigma-lambda-proofs}
    \ell( \lambda, \sigma \mid Y) = \frac{1}{\sigma^2} Y_m^\T K_\lambda(X, X)^{-1} Y_m + \log \det K_\lambda(X, X) + n \log \sigma^2.
  \end{equation}
  For a fixed $\lambda$ it is straightforward to compute that the maximum likelihood estimate of~$\sigma$ is
  \begin{equation*}
    \sigma_\ML(\lambda) = \sqrt{ \frac{Y_m^\T K_\lambda(X, X)^{-1} Y_m }{n} }.
  \end{equation*}
  Plugging this estimate in~\eqref{eq:ell-sigma-lambda-proofs} yields that $\lambda_\ML$ is any minimiser of
  \begin{equation*}
    \ell( \lambda, \sigma_\ML(\lambda) \mid Y ) = n + \log \det K_\lambda(X, X) + n \log ( Y_m^\T K_\lambda(X, X)^{-1} Y_m ) - n \log n.
  \end{equation*}
  The term $n \log ( Y_m^\T K_\lambda(X, X)^{-1} Y_m )$ is bounded as a function of $\lambda$ by \Cref{lemma:data-fit-bounded} and because we have assumed that $Y_m \neq 0$.
  A straightforward adaptation of the proof of~\eqref{eq:mle-infinite} then establishes the claim.
\end{proof}

\subsection{Proof of \Cref{thm:map-estimation}} \label{sec:proof-map-estimation}

\maptheorem*
\begin{proof}
  Recall that the MAP estimator is any minimiser of
  \begin{equation*}
    \ell_\MAP( \lambda \mid Y ) = \frac{1}{2} \ell( \lambda \mid Y ) - \log p(\lambda),
  \end{equation*}
  where $\ell(\lambda \mid Y)$ is the modified log-likelihood function in~\eqref{eq:ell-function}.
  If the data are not $m$-constant, it follows from \Cref{thm:main-theorem} that $\ell( \lambda \mid Y) \to \infty$ as $\lambda \to \infty$.
  Since $p$ is a probability density function, it holds that $-\log p(\lambda) \to \infty$ as $\lambda \to \infty$.
  Consequently,
  \begin{equation} \label{eq:map-estimation-claim-m-constant}
    \lim_{\lambda \to \infty} \ell_\MAP( \lambda \mid Y ) = \infty \quad \text{ and } \quad \lambda_\MAP < \infty
  \end{equation}
  if the data are not $m$-constant.
  Suppose then that the data are $m$-constant.
  The non-negativity of the data-fit term and \Cref{lemma:model-complexity-rate} yield
  \begin{equation*}
    \ell_\MAP( \lambda \mid Y ) \geq -(\alpha - d/2) n \log \lambda + C - \log p(\lambda) = - \log[ \lambda^{(\alpha-d/2)n} p(\lambda)] + C
  \end{equation*}
  for a constant $C$ that does not depend on $\lambda$.
  That $- \log[ \lambda^{(\alpha-d/2)n} p(\lambda)] \to \infty$ as $\lambda \to \infty$ is true by assumption~\eqref{eq:hyperprior-tail-assumption1}.
  Therefore~\eqref{eq:map-estimation-claim-m-constant} holds also when the data are $m$-constant.
\end{proof}

\subsection{Proof of \Cref{thm:generalisation}} \label{sec:proof-of-generalisation}

We begin with a technical lemma which is used to establish some properties of the information functionals.

\begin{lemma} \label{lemma:bounded-lin-funcs} Let $\Omega$ be a vector space and $K$ and $K_\gamma$ two positive-semidefinite kernels on $\Omega$ such that
  \begin{equation*}
    K_\gamma(x,y) = K( \gamma x, \gamma y) \quad \text{ for some } \quad \gamma > 0 \quad \text{ and all } \quad x, y \in \Omega.
  \end{equation*}
  Let $f_\gamma$ be the function $x \mapsto f(\gamma x)$ for any $f \colon \Omega \to \R$ and, given a linear functional $I$, define $I_\gamma$ via $I_\gamma f = I f_\gamma$.
  Then it holds that
  \begin{enumerate}
  \item if $f \in H(K, \Omega)$, then $f_\gamma \in H(K_\gamma, \Omega)$;
  \item if $I$ is a bounded linear functional on $H(K_\gamma, \Omega)$, then $I_\gamma$ is bounded on $H(K, \Omega)$;
  \item if linear functionals $I_1, \ldots, I_n$ are linearly independent on $H(K_\gamma, \Omega)$, then $I_{1,\gamma}, \ldots, I_{n,\gamma}$ are linearly independent on $H(K, \Omega)$.
  \end{enumerate}
\end{lemma}
\begin{proof} By a classical result~\citep[e.g.,][Theorem~3.11]{Paulsen2016} a function $f$ is an element of $H(K, \Omega)$ if and only if there is $c \geq 0$ such that $c^2 K(x, y) - f(x) f(y)$ defines a positive-semidefinite kernel.
  If $f \in H(K, \Omega)$, then $\norm[0]{f}_{H(K, \Omega)}$ equals the smallest $c$ for which this kernel is positive-semidefinite.
  Let $f \in H(K, \Omega)$.
  Therefore $c^2 K(x, y) - f(x) f(y)$ defines a positive-semidefinite kernel for some $c \geq 0$ and
  \begin{equation} \label{eq:c2-f-kernel}
    c^2 K( \gamma x, \gamma y) - f(\gamma x) f(\gamma y) = c^2 K_\gamma(x,y) - f_\gamma(x) f_\gamma(y)
  \end{equation}
  also defines a positive-semidefinite kernel. Consequently, $f_\gamma \in \mathcal{H}(K_\gamma, \Omega)$.
  Furthermore, because $c = \norm[0]{f}_{H(K, \Omega)}$ is the smallest $c$ for which the kernels in~\eqref{eq:c2-f-kernel} are positive-semidefinite, we conclude that $\norm[0]{f_\gamma}_{H(K_\gamma, \Omega)} = \norm[0]{f}_{H(K, \Omega)}$.
  Because $I$ is bounded on $H(K_\gamma, \Omega)$ and $f_\gamma$ is an element of $H(K_\gamma, \Omega)$, there is a constant $C > 0$, which does not depend on $f$, such that
  \begin{equation*}
    \abs[0]{I_\gamma f} = \abs[0]{I f_\gamma} \leq C \norm[0]{f_\gamma}_{H(K_\gamma, \Omega)} = C \norm[0]{f}_{H(K, \Omega)}.
  \end{equation*}
  This concludes the proof of the first two claims.
  The third claim also follows because we have proved that $f_\gamma \in H(K_\gamma, \Omega)$ if $f \in H(K, \Omega)$ and it holds that
  \begin{equation*}
    a_1 I_{1,\gamma}(f) + \cdots + a_n I_{n,\gamma}(f) = a_1 I_{1}(f_\gamma) + \cdots + a_m I_{n}(f_\gamma)
  \end{equation*}
  for any $a_1, \ldots, a_n \in \R$.
\end{proof}

The following proposition is a generalisation of the minimum-norm property for interpolation based on point evaluations that was reviewed in \Cref{sec:rkhs}.
We have not been able to locate a convenient reference and therefore provide a proof.

\begin{proposition} \label{prop:interpolant-norm-bound}
  Let $K$ be a positive-definite kernel on a set $\Omega$. 
  If $f \in H(K, \Omega)$, then
  \begin{equation} \label{eq:rkhs-norm-condition}
    \mathcal{I}(f)^\T K(\mathcal{I}, \mathcal{I})^{-1} \mathcal{I}(f) \leq \norm[0]{f}_{H(K, \Omega)}^2
  \end{equation}
for every $n \geq 1$ and every collection $\mathcal{I} = \{I_1, \ldots, I_n\}$ of $n$ non-trivial linear functionals which are linearly independent and bounded on $H(K, \Omega)$.
\end{proposition}
\begin{proof} Let $s_{\mathcal{I}}f$ be the minimum-norm interpolant to $f$ in $H(K, \Omega)$, which is to say that
  \begin{equation*}
    s_{\mathcal{I}}f = \argmin_{s \in H(K, \Omega)}\Set[\big]{ \norm[0]{s}_{H(K, \Omega)}}{ Is = If \text{ for all } I \in \mathcal{I}}.
  \end{equation*}
  By, for example, Theorem~16.1 in \citet{Wendland2005} or Section~3.3 in \citet{Oettershagen2017} this interpolant is unique and has the explicit form
  \begin{equation*}
    (s_{\mathcal{I}}f)(x) = I(f)^\T K(\mathcal{I}, \mathcal{I})^{-1} K(\mathcal{I}, x) = \sum_{i=1}^n I_i^y (K(x,y)) ( K(\mathcal{I}, \mathcal{I})^{-1} \mathcal{I}(f) )_{i}.
  \end{equation*}
  Recall that we use superscripts to indicate the argument with respect to which a functional is to be applied.
  Because $\inprod{I_j^y K(\cdot, y)}{I_i^x K(x,\cdot)}_{H(K, \Omega)} = I_i^x I_j^y K(x,y)$, from the reproducing property we compute that
  \begin{equation*}
    \begin{split}
      \norm[0]{s_\mathcal{I}f}_{H(K, \Omega)}^2 &= \inprod{ s_\mathcal{I}f }{s_\mathcal{I}f}_{H(K, \Omega)} \\
      &= \sum_{i,j=1}^m \inprod{I_j^y K(\cdot,y)}{I_i^x K(x,\cdot)}_{H(K, \Omega)} ( K(\mathcal{I}, \mathcal{I})^{-1} \mathcal{I}(f) )_{i} ( K(\mathcal{I}, \mathcal{I})^{-1} \mathcal{I}(f) )_{j} \\
      &= \sum_{i,j=1}^n I_i^x I_j^y K(x,y) ( K(\mathcal{I}, \mathcal{I})^{-1} \mathcal{I}(f) )_{i} ( K(\mathcal{I}, \mathcal{I})^{-1} \mathcal{I}(f) )_{j} \\
      &= \mathcal{I}(f)^\T  K(\mathcal{I}, \mathcal{I})^{-1} K(\mathcal{I}, \mathcal{I}) K(\mathcal{I}, \mathcal{I})^{-1} \mathcal{I}(f) \\
      &= \mathcal{I}(f)^\T K(\mathcal{I}, \mathcal{I})^{-1} \mathcal{I}(f).
      \end{split}
  \end{equation*}
  If $f \in H(K, \Omega)$, the condition $Is = If$ for all $I \in \mathcal{I}$ is trivially satisfied by $s = f$ and therefore $\norm[0]{s_\mathcal{I}f}_{H(K, \Omega)}^2 = \mathcal{I}(f)^\T K(\mathcal{I}, \mathcal{I})^{-1} \mathcal{I}(f) \leq \norm[0]{f}_{H(K, \Omega)}^2$.
\end{proof}

These preliminaries suffice to prove Theorem~\ref{thm:generalisation}.\\

\gentheorem*
\begin{proof}
  Let $L_{i,\lambda}$ be linear functionals that are defined as $L_{i,\lambda} f = L_if(\cdot/g(\lambda))$ for $f \in H(K, \Omega)$.
  These functionals are well-defined because $f(\cdot/g(\lambda)) \in H(K_{1/\lambda}, \Omega) \subset F(\Omega)$ by \Cref{lemma:bounded-lin-funcs}.
  Moreover, by \Cref{lemma:bounded-lin-funcs} and the assumption that there exists $\lambda_1 > 0$ such that $g(\lambda_1) = 1$, the linear functionals $\mathcal{L}_\lambda = \{L_{1,\lambda}, \ldots, L_{n,\lambda} \}$ are linearly independent and bounded on $H(K, \Omega) = H(K_{\lambda_1}, \Omega)$ for every $\lambda > 0$.
  Because $g$ is continuous and $g(\lambda) \to \infty$ as $\lambda \to \infty$, there is $\lambda_0 > 0$ such that $g(\lambda) \geq 1$ for all $\lambda \geq \lambda_0$.
  The assumptions on $\Omega_b$ imply that $\Omega_b \subset \Omega_{b,\lambda} = \Set{g(\lambda) x}{ x \in \Omega_b}$ if $\lambda \geq \lambda_0$.
  Let $f_1, f_2 \in H(K, \Omega)$ be such that $f_1 = f_2$ on $\Omega_b$.
  Then $f_1(\cdot/g(\lambda)) = f_2(\cdot/g(\lambda))$ on $\Omega_b \subset \Omega_{b,\lambda}$ and thus 
  \begin{equation*}
    L_{i,\lambda} f_1 = L_{i} \bigg[ f_1\bigg(\frac{\cdot}{g(\lambda)} \bigg) \bigg] = L_{i} \bigg[ f_2\bigg(\frac{\cdot}{g(\lambda)} \bigg) \bigg] = L_{i,\lambda} f_2
  \end{equation*}
  by the assumption that $L_if_1 = L_if_2$ for every $i \in \{1, \ldots, n\}$ and all $f_1, f_2 \in F(\Omega)$ such that $f_1 = f_2$ on $\Omega_b$.
  Therefore the functionals $\mathcal{L}_\lambda$ admit well-defined, bounded, and linearly independent restrictions on $H(K, \Omega_b)$ for every $\lambda \geq \lambda_0$.
  Let $f \equiv c$ for $c \in \R$ be the constant function such that $Y_{\mathcal{L}, m} = \mathcal{L}(f)$.
  Then
  \begin{equation*}
    Y_{\mathcal{L},m}^\T K_\lambda( \mathcal{L}, \mathcal{L} )^{-1} Y_{\mathcal{L}, m} = \mathcal{L}(f)^\T K_\lambda( \mathcal{L}, \mathcal{L} )^{-1} \mathcal{L}(f) = \mathcal{L}_\lambda(f)^\T K( \mathcal{L}_\lambda, \mathcal{L}_\lambda )^{-1} \mathcal{L}_\lambda(f)
  \end{equation*}
  since $f(\cdot / g(\lambda)) = f$ for constant functions. The assumption that $f \in H(K, \Omega_b)$ and \Cref{prop:interpolant-norm-bound} then imply that
  \begin{equation} \label{eq:term1-bound1}
    \sup_{ \lambda \geq \lambda_0} Y_{\mathcal{L},m}^\T K_\lambda( \mathcal{L}, \mathcal{L} )^{-1} Y_{\mathcal{L}, m} \leq \norm[0]{f}_{H(K, \Omega_b)}^2.
  \end{equation}
  By \Cref{ass:general}, the data-fit term $Y_{\mathcal{L},m}^\T K_\lambda( \mathcal{L}, \mathcal{L} )^{-1} Y_{\mathcal{L}, m}$ is continuous in $\lambda$ (which can be proved similarly to \Cref{lemma:continuity}) and satisfies
  \begin{equation*}
    \limsup_{\lambda \to 0} Y_{\mathcal{L},m}^\T K_\lambda( \mathcal{L}, \mathcal{L} )^{-1} Y_{\mathcal{L}, m} \leq \frac{1}{\rho} \norm[0]{ Y_{\mathcal{L}, m}}^2 < \infty,
  \end{equation*}
  where $\rho = \liminf_{ \lambda \to 0} e_\textup{min} (K_\lambda( \mathcal{L}, \mathcal{L} )) > 0$.
  From this and~\eqref{eq:term1-bound1} it follows that there is $a \geq 0$ such that 
  \begin{equation} \label{eq:term1-bound}
    0 \leq Y_{\mathcal{L},m}^\T K_\lambda( \mathcal{L}, \mathcal{L} )^{-1} Y_{\mathcal{L}, m} \leq a
  \end{equation}
  for all $\lambda > 0$, which is a generalisation of \Cref{lemma:data-fit-bounded}.
  Finally, Assumption~\ref{ass:general} implies that $\det K_\lambda( \mathcal{L}, \mathcal{L} ) \to 0$ if and only if $\lambda \to \infty$ since $\det K_\lambda( \mathcal{L}, \mathcal{L} )$ is a continuous function of $\lambda$ and positive for every $\lambda > 0$ by positive-definiteness of $K_\lambda(\mathcal{L}, \mathcal{L}) = K(\mathcal{L}_\lambda, \mathcal{L}_\lambda)$.
  We thus conclude from~\eqref{eq:term1-bound} that
  \begin{equation*}
    \ell( \lambda \mid Y_\mathcal{L} ) = Y_{\mathcal{L},m}^\T K_\lambda( \mathcal{L}, \mathcal{L} )^{-1} Y_{\mathcal{L}, m} + \log \det K_\lambda( \mathcal{L}, \mathcal{L} )
  \end{equation*}
  is a sum two $\lambda$-continuous terms, the first of which is bounded while the second tends to negative infinity if and only if $\lambda \to \infty$.
  Therefore $\lambda_{\ML(\mathcal{L})} = \argmin_{\lambda > 0} \ell( \lambda \mid Y_\mathcal{L} ) = \infty$.
\end{proof}

\subsection{Proof of \Cref{thm:multiple-lengthscales}} \label{sec:multiple-lambda-proof}

\multiplelengthscales*
\begin{proof}
  Because both the kernel $K_\theta$ and the covariates $X$ have product forms, we may write the covariance matrix as
  \begin{equation*}
    K_\theta(X, X) = K_{1,\lambda_1}(X_1, X_1) \otimes \cdots \otimes K_{d,\lambda_d}(X_d, X_d),
  \end{equation*}
  where $\otimes$ denotes the Kronecker product.
  We may assume that $p = 1$ without loss of generality.
  This means that we may write
  \begin{equation*}
    Y_m = Y - m(X) = 1_{n_1} \otimes Y_m',
  \end{equation*}
  where $1_{n_1} = (1, \ldots, 1) \in \R^{n_1}$ and $Y_m' \in \R^{n'}$ for $n' = n_2 \times \cdots \times n_d$ is a certain vector.
  Let $A = K_{2, \lambda_2}(X_2, X_2) \otimes \cdots \otimes K_{d,\lambda_d}(X_d, X_d) \in \R^{n' \times n'}$.
  The properties of the Kronecker product yield
  \begin{equation*}
    Y_m^\T K_\theta(X, X)^{-1} Y_m = \big[ 1_{n_1}^\T K_{1,\lambda_1}(X_1, X_1)^{-1} 1_{n_1} \big] \times \big[ (Y_m')^\T A^{-1} Y_m' \big]
  \end{equation*}
  and
  \begin{equation*}
    \det K_\theta(X, X) = \big[\det K_{1,\lambda_1}(X_1, X_1) \big]^{n'} \times [ \det A ]^{n_1}.
  \end{equation*}
  Because neither $Y_m'$ nor $A$ depends on $\lambda_1$, we may proceed as in the proof of \Cref{thm:main-theorem} to show that $1_{n_1}^\T K_{1,\lambda_1}(X_1, X_1)^{-1} 1_{n_1}$ is a bounded function of $\lambda_1$ while $\log \det K_{1,\lambda_1}(X_1, X_1)$ tends to negative infinity as $\lambda_1 \to \infty$.  
\end{proof}

\subsection{Proof of \Cref{thm:smooth-kernels}} \label{sec:smooth-proof}

Let $d = 1$ and $K(x, y) = \Phi(x -y)$ for an integrable function $\Phi \colon \R \to \R$ that is infinitely differentiable in a neighbourhood of the origin.
Define the diagonal matrix
\begin{equation*}
  \Delta_\lambda = \mathrm{diag}(1, \lambda^{-1}, \ldots, \lambda^{n-1})
\end{equation*}
and the \emph{Vandermonde matrix}
\begin{equation*}
  V(X) =
  \begin{pmatrix}
    1 & x_1 & \cdots & x_1^{n-1} \\
    \vdots & \vdots & \ddots & \vdots \\
    1 & x_n & \cdots & x_n^{n-1}
  \end{pmatrix}.
\end{equation*}
It is a standard result in polynomial interpolation that the Vandermonde matrix is non-singular when $x_i$ are distinct.
Recall that $W$ is the Wronskian defined in~\eqref{eq:wronskian}.
The Wronskian is non-singular if the Fourier transform of $\Phi$ is positive~\citep[Lemma~3.3]{LeeYoonYoon2007}.
To see this, note first that
\begin{equation*}
    (-1)^{j} (-\mathrm{i})^{i+j} \frac{\partial^{i+j}}{\partial v^i \partial w^j} K(v, w) \Bigl|_{\substack{v = 0 \\ w = 0}} = (-\mathrm{i})^{i+j} \Phi^{(i+j)}(0) = \frac{1}{2\pi} \int_\R \xi^{i+j} \widehat{\Phi}(\xi) \dif \xi
\end{equation*}
and observe then that
\begin{equation*}
  \sum_{i=0}^{n-1} \sum_{j=0}^{n-1} a_i a_j \Phi^{(i+j)}(0) = \frac{1}{2\pi} \int_\R \widehat{\Phi}(\xi) \Bigg( \sum_{i=1}^n a_i \xi^i \Bigg)^2 \dif \xi
\end{equation*}
is positive for any $n \in \N$ and any non-zero $a = (a_0, \ldots, a_{n-1}) \in \R^n$.
It follows from these equations that the Wronskian is non-singular because it can be written as a product of a positive-definite matrix with elements $\Phi^{(i+j)}(0)$ and two non-singular diagonal matrices.\\

\smooththeorem*
\begin{proof}
We proceed as in the proof of \Cref{eq:mle-infinite} in \Cref{sec:proof-main} and conclude that to prove the claim it is sufficient to show that the limit $\lim_{\lambda \to \infty} Y_m^\T K_\lambda(X, X)^{-1} Y_m$ exists and is finite.
To prove that this is so we use Equation~(32) in \citet{BarthelmeUsevich2021}, which states that
\begin{equation*}
  K_\lambda(X, X) = V(X) \Delta_\lambda W \Delta_\lambda V(X)^\T + \lambda^{-n} \big[ V(X) \Delta_\lambda W_{1,\lambda} + W_{2,\lambda} \Delta_\lambda V(X)^\T \big] + \lambda^{-2n} W_{3,\lambda},
\end{equation*}
where the matrices $W_{1,\lambda}$, $W_{2,\lambda}$, and $W_{3,\lambda}$ are bounded as $\lambda \to \infty$.
We may now use this equation to write
\begin{equation} \label{eq:data-fit-barthelme}
  Y_m^\T K_\lambda(X, X)^{-1} Y_m = Y_m^\T V(X)^{-\T} \Delta_\lambda^{-1} ( W + A_\lambda)^{-1} \Delta_\lambda^{-1} V(X)^{-1} Y_m,
\end{equation}
where
\begin{equation*}
    A_\lambda = \lambda^{-n} \big[ W_{1,\lambda} V(X)^{-\T} \Delta_\lambda^{-1} + \Delta_\lambda^{-1} V(X)^{-1} W_{2,\lambda} \big] + \lambda^{-2n} \Delta_\lambda^{-1} V(X)^{-1} W_{3,\lambda} V(X)^{-\T} \Delta_\lambda^{-1}.
\end{equation*}
Since $\Delta_\lambda^{-1} = \mathcal{O}(\lambda^{n-1})$, we conclude that $A_\lambda = \mathcal{O}(\lambda^{-1})$.
Furthermore, from the assumption that $Y_m = (c, \ldots, c)$ for some $c \in \R$ it follows that $V(X)^{-1} Y_m = (c, 0, \ldots, 0) \in \R^n$.
Consequently, $\Delta_\lambda^{-1} V(X)^{-1} Y_m = (c, 0, \ldots, 0) = D_0$.
Therefore~\eqref{eq:data-fit-barthelme} yields
\begin{equation*}
  Y_m^\T K_\lambda(X, X)^{-1} Y_m = D_0^\T ( W + A_\lambda)^{-1} D_0 \to D_0^\T W^{-1} D_0 < \infty \quad \text{ as } \quad \lambda \to \infty,
\end{equation*}
where the Wronskian $W$ is, as noted earlier, non-singular.
\end{proof}

\section*{Acknowledgements}

TK was supported by the Academy of Finland postdoctoral researcher grant \#338567 ``Scalable, adaptive and reliable probabilistic integration''.
CJO was supported by the Alan Turing Institute, United Kingdom, and the Engineering \& Physical Sciences Research Council grant EP/WO19590/1.
We are grateful to Philipp Hennig, Motonobu Kanagawa, and Fran\c{c}ois-Xavier Briol for helpful general comments and to Simon Barthelm\'e and Konstantin Usevich for pointing out how to prove \Cref{thm:smooth-kernels}.
Comments by an anonymous reviewer served as inspiration for \Cref{sec:multiple-lengthscales,sec:other-parameters}.

\end{document}